\documentclass{amsart}

\usepackage{amsmath,amsthm,amssymb,bm}
\usepackage{hyperref}
\usepackage{cleveref}
\usepackage{graphicx,color}
\usepackage{tikz}
\usetikzlibrary{decorations.markings}
\usetikzlibrary{tikzmark}
\usetikzlibrary{arrows.meta}
\usepackage{ytableau}
\usepackage{cases}
\usepackage{subcaption}
\usepackage{hhline}
\numberwithin{equation}{section}
\newtheorem{thm}{Theorem}[section]
\newtheorem{lem}[thm]{Lemma}
\newtheorem{prop}[thm]{Proposition}
\newtheorem{cor}[thm]{Corollary}

\theoremstyle{definition}

\newtheorem{defn}[thm]{Definition}

\newtheorem{remark}[thm]{Remark}
\newtheorem*{remark*}{Remark}

\newcommand\PP{\mathbb{P}}
\newcommand\FSVT{\operatorname{FSVT}}
\newcommand\DPar{\operatorname{DPar}}
\newcommand\Par{\operatorname{Par}}
\newcommand\RPar{\operatorname{RPar}}

\newcommand\multinom[2]{\left(\! \binom{#1}{#2} \!\right)}
\newcommand\unmarked{\operatorname{unmarked}}
\newcommand\marked{\operatorname{marked}}
\newcommand\MMSVT{\operatorname{MMSVT}}

\newcommand\MRBT{\operatorname{MRBT}}

\newcommand\qand{\quad\mbox{and}\quad}
\newcommand\NN{\mathbb{N}}

\newcommand\col{\operatorname{col}}

\newcommand\vr{\bm{r}}
\newcommand\vs{\bm{s}}
\newcommand\vx{\bm{x}}

\newcommand\va{\bm{\alpha}}
\newcommand\vb{\bm{\beta}}

\newcommand\vt{\bm{t}}

\newcommand\tg{\tilde{g}}
\newcommand\tG{\tilde{G}}

\newcommand\row{\operatorname{row}}

\newcommand\ZZ{\mathbb{Z}}

\newcommand\wt{\operatorname{wt}}

\newcommand\MRPP{{\operatorname{MRPP}}}
\newcommand\RMRPP{{\operatorname{RMRPP}}}
\newcommand\BMRPP{{\operatorname{BMRPP}}}

\newcommand\lm{{\lambda/\mu}}

\newcommand\lmc{{\lambda'/\mu'}}

\title{Refined canonical stable Grothendieck polynomials and their duals,
Part 2}

\author{Byung-Hak Hwang}
\address{School of Mathematics, Korea Institute for Advanced Study,
  South Korea}
\email{byunghakhwang@gmail.com}
\author{Jihyeug Jang}
\address{Department of Mathematics, Sungkyunkwan University, Suwon,
  South Korea}
\email{jihyeugjang@gmail.com}
\author{Jang Soo Kim}
\address{Department of Mathematics, Sungkyunkwan University, Suwon,
  South Korea}
\email{jangsookim@skku.edu}
\author{Minho Song}
\address{Department of Mathematics, Sungkyunkwan University, Suwon,
  South Korea}
\email{smh3227@skku.edu}
\author{U-Keun Song}
\address{Department of Mathematics, Sungkyunkwan University, Suwon,
  South Korea}
\email{sukeun319@gmail.com}

\date{\today}

\begin{document}

\begin{abstract}
  This paper is the sequel of the paper under the same title with part
  1, where we introduced refined canonical stable Grothendieck
  polynomials and their duals with two families of infinite
  parameters. In this paper we give combinatorial interpretations for
  these polynomials using generalizations of set-valued tableaux and
  reverse plane partitions, respectively. Our results extend to their
  flagged and skew versions.
\end{abstract}

\maketitle

\section{Introduction}
\label{sec:introduction}

This paper is a continuation of the study of refined canonical stable
Grothendieck polynomials \( G_\lambda(\vx;\va,\vb) \) and their duals
\( g_\lambda(\vx;\va,\vb) \) introduced in the previous paper
\cite{HJKSS2024}. The main goal of this paper is to give
combinatorial models for \( G_\lambda(\vx;\va,\vb) \) and
\( g_\lambda(\vx;\va,\vb) \) using generalizations of set-valued
tableaux and reverse plane partitions, respectively. In doing so, we
also generalize \( G_\lambda(\vx;\va,\vb) \) and
\( g_\lambda(\vx;\va,\vb) \) to their flagged and skew versions.

We briefly review known results on Grothendieck polynomials relevant
to the current paper. For a more detailed history of Grothendieck
polynomials, see the introduction of the previous paper
\cite{HJKSS2024}.

In 1982, Lascoux and Sch\"{u}tzenberger \cite{LS1982} introduced
Grothendieck polynomials, which are \( K \)-theoretic analogues of
Schubert polynomials. Fomin and Kirillov \cite{FK1996} added a
parameter \( \beta \) to Grothendieck polynomials and studied stable
limits of \( \beta \)-Grothendieck polynomials. The stable
Grothendieck polynomials \( G_\lambda(\vx) \) of Grassmannian type are
indexed by partitions and form a basis of (a completion of) the
symmetric function space.

It is a classical result that the Schur function \( s_\lambda(\vx) \)
has a combinatorial model using semistandard Young tableaux. In 2002,
Buch~\cite{Buch2002} found a similar combinatorial model for
\( G_\lm(\vx) \) by introducing set-valued tableaux. The dual
Grothendieck polynomials \( g_\lambda(\vx) \) are dual to
\( G_\lambda(\vx) \) with respect to the Hall inner product. In 2007,
Lam--Pylyavskyy~\cite{LP2007} found a combinatorial model for
\( g_\lm(\vx) \) using reverse plane partitions. Since then,
Grothendieck polynomials \( G_\lambda(\vx) \) and their duals
\( g_\lambda(\vx) \) have been shown to have many interesting
combinatorial properties and generalizations.

Yeliussizov~\cite{Yeliussizov2017} introduced canonical stable
Grothendieck polynomials \( G^{(\alpha,\beta)}_\lambda(\vx) \) and
their duals \( g^{(\alpha,\beta)}_\lambda(\vx) \) with two parameters
\( \alpha \) and \( \beta \), which behave nicely under the
\( \omega \) involution on the symmetric function space. Motivated by
their study of Brill--Noether varieties in \cite{Chan2021},
Chan--Pflueger~\cite{CP21:grothendieck} introduced Grothendieck
polynomials \( RG_\lambda(\vx;\vb) \) with a family
\( \vb = (\beta_1,\beta_2,\dots) \) of infinite parameters. In a
completely different context, Galashin--Grinberg--Liu~\cite{GGL2016}
introduced dual Grothendieck polynomials \( \tg_{\lambda}(\vx;\vb) \)
with infinite parameters \( \vb \), which turn out to be dual to
\( RG_\lambda(\vx;\vb) \) with respect to the Hall inner product.

In the previous paper \cite{HJKSS2024} we introduced refined
canonical stable Grothendieck polynomials \( G_\lambda(\vx;\va,\vb) \)
and their duals \( g_\lambda(\vx;\va,\vb) \) with two families
\( \va \) and \( \vb \) of infinite parameters, which generalize all
variations of Grothendieck polynomials mentioned so far. In this paper
we find combinatorial models for \( G_\lambda(\vx;\va,\vb) \) and
\( g_\lambda(\vx;\va,\vb) \) using generalizations of set-valued
tableaux and reverse plane partitions, respectively. Our results, in
fact, apply to further generalizations of \( G_\lambda(\vx;\va,\vb) \)
and \( g_\lambda(\vx;\va,\vb) \) described below.

There is another way to generalize Grothendieck polynomials by
imposing ``flag conditions'' on their tableau models, where a flag
condition means that the entries in a row or column are bounded by
specific numbers. Gessel and Viennot \cite{LGV} found a
Jacobi--Trudi-like formula for flagged Schur functions, which are
generating functions for flagged semistandard Young tableaux.
Wachs~\cite{Wachs_1985} gave another proof of this result using
natural recursions coming from flagged semistandard Young tableaux.
Knutson, Miller, and Yong \cite{Knutson2009} showed that Grothendieck
polynomials associated with vexillary permutations are generating
functions for flagged set-valued tableaux.
Matsumura~\cite{Matsumura_2018} introduced flagged Grothendieck
polynomials \( G_{\lambda/\mu,f/g}(\vx) \) as generating functions for
flagged set-valued tableaux and proved their Jacobi--Trudi-like
formula. For the dual case, Grinberg~\cite{Grinberg_private} first
considered flagged versions
\( \tg^{\row(\vr,\vs)}_{\lambda/\mu}(\vx;\vb) \) of
\( \tg_{\lambda/\mu}(\vx;\vb) \), which were further studied by
Kim~\cite{Kim_JT22}.

In this paper, we generalize \( G_\lambda(\vx;\va,\vb) \) and
\( g_\lambda(\vx;\va,\vb) \) to their flagged and skew versions
\( G_\lm^{\row(\vr, \vs)}(\vx;\va,\vb) \),
\( G_{\lambda'/\mu'}^{\col(\vr, \vs)}(\vx;\va,\vb) \),
\( g_\lm^{\row(\vr, \vs)}(\vx;\va,\vb) \), and
\( g_{\lambda'/\mu'}^{\col(\vr, \vs)}(\vx;\va,\vb) \). These
generalize all variations of Grothendieck polynomials mentioned above
as shown in the following table. Here \( \bm{0} = (0,0,\dots) \),
\( \bm{1}=(1,1,\dots) \), \( \va_0 = (\alpha,\alpha,\dots) \) and
\( \vb_0 = (\beta,\beta,\dots) \).

\begin{center}
  \begin{tabular}{|c|c|c|}
    \hline
    Variations of Grothendieck polynomials & introduced in & how to specialize \\ \hhline{|=|=|=|}
    \( G_{\nu/\lambda}(\vx) \) & Buch~\cite{Buch2002} &
      \( G_{\nu/\lambda}(\vx;\bm{0},\bm{1}) \)  \\ \hline
    \( g_{\lambda/\mu}(\vx) \) & Lam--Pylyavskyy~\cite{LP2007} &
      \( g_{\lambda/\mu}(\vx;\bm{0},\bm{1}) \) \\  \hline
    \( G^{(\alpha,\beta)}_\lambda(\vx) \) & Yeliussizov~\cite{Yeliussizov2017} &
        \( G_{\lambda}(\vx;\va_0,-\vb_0) \) \\  \hline
    \( g^{(\alpha,\beta)}_\lambda(\vx) \) & Yeliussizov~\cite{Yeliussizov2017} &
        \( g_{\lambda}(\vx;-\va_0,\vb_0) \) \\  \hline
    \( RG_\lambda(\vx;\vb) \) & Chan--Pflueger~\cite{CP21:grothendieck} &
      \( G_\lambda(\vx;\bm{0},\vb) \) \\  \hline
    \( \tg_{\lambda/\mu}(\vx;\vb) \) & Galashin--Grinberg--Liu~\cite{GGL2016} &
      \( g_{\lambda/\mu}(\vx;\bm{0},\vb) \) \\  \hline
    \( G_{\lambda/\mu,f/g}(\vx) \) & Matsumura~\cite{Matsumura_2018} &
      \( G^{\row(g,f)}_{\lambda/\mu}(\vx;\bm{0},-\vb_0) \) \\  \hline
    \( \tg^{\row(\vr,\vs)}_{\lambda/\mu}(\vx;\vb) \) & Grinberg \cite{Grinberg_private}, Kim~\cite{Kim_JT22} &
      \( g^{\row(\vr,\vs)}_{\lambda/\mu}(\vx;\bm{0},\vb) \) \\  \hline
  \end{tabular}
\end{center}

In what follows we illustrate the main results in this paper. To do
this we recall the definitions of \( G_\lambda(\vx_n;\va,\vb) \) and
\( g_\lambda(\vx_n;\va,\vb) \) in our previous paper
\cite{HJKSS2024} and their Jacobi--Trudi-like formulas; see
Section~\ref{sec:preliminary} for the definitions of undefined terms.

\begin{defn}\label{def:bi_alt_G}
  Let \( \lambda \) be a partition \( \lambda \) with at most \( n \)
  parts. The \emph{refined canonical stable Grothendieck polynomial}
  \( G_\lambda(\vx_n;\va,\vb) \) and the \emph{refined dual canonical
    stable Grothendieck polynomial} \( g_\lambda(\vx_n;\va,\vb) \) are
  defined by
\begin{align}
  \label{eq:defG}
    G_\lambda(\vx_n;\va,\vb) 
    &= \frac{\det
      \left(x_j^{\lambda_i+n-i}
      \dfrac{(1-\beta_1 x_j)\cdots (1-\beta_{i-1} x_j)}
            {(1-\alpha_1 x_j)\cdots (1-\alpha_{\lambda_i} x_j)}
      \right)_{1\le i,j\le n}}
      {\prod_{1\le i<j\le n}(x_i-x_j)},\\
      \label{eq:defg}
    g_\lambda(\vx_n;\va,\vb)
   & = \frac{\det(h_{\lambda_i+n-i}[x_j-A_{\lambda_i-1}+B_{i-1}])_{1\le i,j\le n}}
      {\prod_{1\le i<j\le n}(x_i-x_j)}.
  \end{align}
\end{defn}

In the first part \cite{HJKSS2024} we proved that
\( G_\lambda(\vx_n;\va,\vb) \) and \( g_\lambda(\vx_n;\va,\vb) \) have
Jacobi--Trudi-like formulas.

\begin{thm}[\cite{HJKSS2024}]\label{thm:JT_ab_intro}
  For a partition \( \lambda \) with at most \( n \) parts, we have
  \begin{align}
\label{eq:4}    G_\lambda(\vx_n;\va,\vb)
     &=  C_n \det \left( h_{\lambda_i-i+j}[X_n\ominus (A_{\lambda_i}-B_{i-1}+B_j)] \right)_{i,j=1}^n,\\
\label{eq:5}    g_\lambda(\vx_n;\va,\vb)
    &= \det \left( h_{\lambda_i-i+j}[X_n-A_{\lambda_i-1}+B_{i-1}-B_{j-1}] \right)_{i,j=1}^n,
  \end{align}
  where \( C_n = \prod_{i,j=1}^n (1-\beta_i x_j) \) and \( h_n[Y\ominus Z] = \sum_{a-b=n} h_a[Y] h_b[Z] \). 
\end{thm}

The main results in this paper are the following combinatorial
interpretations for \( G_\lambda(\vx_n;\va,\vb) \) and
\( g_\lambda(\vx_n;\va,\vb) \) using ``marked multiset-valued
tableaux'' and ``marked reverse plane partitions'', which are
generalizations of set-valued tableaux and reverse plane partitions,
respectively; see Sections~\ref{sec:flagged G} and \ref{sec:flagged g}
for their precise definitions.

\begin{thm}\label{thm:comb_intro}
  For a partition \( \lambda \), we have
  \begin{align}
\label{eq:6}    G_\lambda(\vx;\va,\vb) &= \sum_{T\in \MMSVT(\lambda)} \wt(T),\\
\label{eq:7}    g_\lambda(\vx;\va,\vb) &= \sum_{T\in\MRPP(\lambda)} \wt(T),
  \end{align}
  where \( \MMSVT(\lambda) \) and \( \MRPP(\lambda) \) are the sets of
  marked multiset-valued tableaux and marked reverse plane partitions
  of shape \( \lambda \), respectively.
\end{thm}

Our strategy to prove Theorem~\ref{thm:comb_intro} is as follows.
Instead of proving \eqref{eq:6} and \eqref{eq:7} directly, we consider
their flagged versions \( G_\lm^{\row(\vr, \vs)}(\vx;\va,\vb) \) and
\( g_\lm^{\row(\vr, \vs)}(\vx;\va,\vb) \), see
Definitions~\ref{defn:Grow} and \ref{def:1}, for which we can use
inductive arguments in the same vein as in Wachs \cite{Wachs_1985}.

Using computer experiments, we were able to find suitable
Jacobi--Trudi-like formulas for
\( G_\lm^{\row(\vr, \vs)}(\vx;\va,\vb) \) and
\( g_\lm^{\row(\vr, \vs)}(\vx;\va,\vb) \) that generalize \eqref{eq:4}
and \eqref{eq:5}, respectively; see Theorem~\ref{thm:main_G_intro} below. Thanks to
the flag conditions, \( G_\lm^{\row(\vr, \vs)}(\vx;\va,\vb) \) and
\( g_\lm^{\row(\vr, \vs)}(\vx;\va,\vb) \) have natural recurrences. We
then prove the Jacobi--Trudi-like formulas by showing that the
determinants also have the same recurrences. Although our proofs
require many technical lemmas, they are elementary and we need nothing
more complex than basic techniques in linear algebra such as row
operations.

Our methods also work for column-flagged versions. In summary we prove
the following Jacobi--Trudi-like formulas. Note that \eqref{eq:8} and
\eqref{eq:10} generalize \eqref{eq:4} and \eqref{eq:5}, hence imply
\eqref{eq:6} and \eqref{eq:7}, respectively. We also note that
\eqref{eq:9} is proved under a weaker condition; see
Theorem~\ref{thm:main_G_col}.

\begin{thm}
  \label{thm:main_G_intro}
  Let \( \lambda \) and \( \mu \) be partitions with
  at most \( n \) parts and \( \mu\subseteq\lambda \).
  Let \( \vr=(r_1,\dots,r_n) \) and \( \vs=(s_1,\dots,s_n) \) be positive
  integer sequences of length \( n \).
  Suppose that \( r_i\le r_{i+1} \) and \( s_i\le s_{i+1} \)
  whenever \( \mu_{i}<\lambda_{i+1} \). Then
  \begin{align}
\label{eq:8}    G_\lm^{\row(\vr, \vs)}(\vx;\va,\vb) 
    &= C \det\left(   h_{\lambda_i-\mu_j-i+j}
      [X_{[r_j,s_i]}\ominus(A_{\lambda_i}-A_{\mu_j}-B_{i-1} + B_{j})]\right)_{i,j=1}^n,\\
\label{eq:9}    G_{\lambda'/\mu'}^{\col(\vr, \vs)}(\vx;\va,\vb) 
    &= D \det\left(   e_{\lambda_i-\mu_j-i+j}
    [X_{[r_j,s_i]}\ominus(A_{i-1}-A_{j}-B_{\lambda_i} + B_{\mu_j})]\right)_{i,j=1}^n,\\
\label{eq:10}    g_\lm^{\row(\vr, \vs)}(\vx;\va,\vb)
    &= \det\left(   h_{\lambda_i-\mu_j-i+j}
      [X_{[r_j,s_i]}-A_{\lambda_i-1}+A_{\mu_j}+B_{i-1} - B_{j-1}]\right)_{i,j=1}^n,\\
\label{eq:11}    g_{\lambda'/\mu'}^{\col(\vr, \vs)}(\vx;\va,\vb)
    &= \det\left(   e_{\lambda_i-\mu_j-i+j}
      [X_{[r_j,s_i]}-A_{i-1}+A_{j-1}+B_{\lambda_i-1} - B_{\mu_j}]\right)_{i,j=1}^n,
  \end{align}
  where 
  \[
    C = \prod_{i=1}^{n} \prod_{l=r_i}^{s_i} (1-\beta_i x_l), \qquad
    D = \prod_{i=1}^{n} \prod_{l=r_i}^{s_i} (1-\alpha_i x_l)^{-1}.
  \]
\end{thm}

The rest of this paper is organized as follows. In
Section~\ref{sec:preliminary} we recall necessary definitions and
lemmas. In Section~\ref{sec:flagged G} we prove \eqref{eq:8} and
\eqref{eq:9}. We also show that our result implies the
Jacobi--Trudi-like formula for \( G_{\lambda/\mu,f/g}(\vx) \) due to
Matsumura~\cite{Matsumura_2018}. In Section~\ref{sec:flagged g} we
prove \eqref{eq:10} and \eqref{eq:11}.

\section{Preliminaries}
\label{sec:preliminary}

In this section, we briefly recall necessary definitions and basic
lemmas from the first part \cite{HJKSS2024}. Throughout this
paper, we denote by \( \ZZ \) (resp.~\( \NN \), \( \PP \)) the set of
integers (resp.~nonnegative integers, positive integers). For
\( n\in\NN \) we write \( [n]=\{1,2,\dots,n\} \). For a statement
\( p \), we define \( \chi(p) \) to be \( 1 \) if \( p \) is true and
\( 0 \) otherwise.

\medskip

A \emph{partition} is a weakly decreasing sequence
\( \lambda=(\lambda_1,\dots,\lambda_\ell) \) of positive integers. We
call each \( \lambda_i \) a \emph{part} of \( \lambda \) and write
\( \ell(\lambda) = \ell \) for the number of parts. If \( i>\ell(\lambda) \), we define
\( \lambda_i=0 \). The set of all partitions is denoted by \( \Par \)
and the set of partitions with at most \( n \) parts is denoted by
\( \Par_n \).

Let \(\lambda=(\lambda_1,\dots,\lambda_n)\in \Par_n\). We will
identify the partition \(\lambda\) with its \emph{Young diagram},
i.e.,
\( \lambda= \{ (i,j)\in \ZZ^2 : 1\leq i \leq n, 1 \leq j \leq
\lambda_i \} \). Each \((i,j) \in \lambda\) is called a \emph{cell} of
\(\lambda\). The \emph{transpose} \( \lambda' \) of \( \lambda \) is
the partition whose Young diagram is
\( \{ (j,i)\in \ZZ^2 : 1\leq i \leq n, 1 \leq j \leq \lambda_i \} \).
We use ``the English notation'' to visualize partitions; see
Figure~\ref{fig:YD}.

\begin{figure}
    \ytableausetup{mathmode, boxsize=1em}
  \centering
  \ydiagram{4,3,1} \qquad\qquad\qquad\qquad \ydiagram{3,2,2,1}
  \caption{The Young diagram of \( \lambda=(4,3,1) \) on the left and its transpose \( \lambda'=(3,2,2,1) \)
    on the right.}
  \label{fig:YD}
\end{figure}

For two partitions \( \lambda \) and \( \mu \), we write
\( \mu\subseteq\lambda \) if the Young diagram of \( \mu \) is
contained in that of \( \lambda \). In this case the \emph{skew shape}
\( \lm \) is defined to be the set-theoretic difference
\( \lambda-\mu \) of the Young diagrams. The \emph{size} of \( \lm \)
is the number \( |\lm| \) of cells in \( \lm \). The partition
\( \lambda \) will also be considered as the skew shape
\( \lambda/\emptyset \), where \( \emptyset \) is the empty partition.

\medskip

Let \( \vx=(x_1,x_2,\dots) \) be a set of formal variables. A formal
power series \( f(\vx) \) in \( \vx \) is called a \emph{symmetric
  function} if it is invariant under permuting the variables. Let
\( \vx_n = (x_1,\dots,x_n) \) and define \( f(\vx_n) \) to be the
formal power series \( f(\vx) \) with substitution \( x_i=0 \) for
\( i>n \). In the literature, symmetric functions often mean symmetric
formal power series of bounded degree, but in this paper we do not
require the bounded degree condition. We refer the reader to
\cite[Chapter 7]{EC2} and \cite{Macdonald} for more details on
symmetric functions.

The \emph{complete homogeneous
  symmetric function} \( h_n(\vx) \) and the \emph{elementary
  symmetric function} \( e_n(\vx) \) are defined by
\[
  h_n(\vx) = \sum_{i_1\le\dots\le i_n} x_{i_1} \cdots x_{i_n}, \qquad 
  e_n(\vx) = \sum_{i_1<\dots<i_n} x_{i_1} \cdots x_{i_n},
\]
where we use the convention that \( h_0(\vx) = e_0(\vx) = 1 \) and
\( h_n(\vx) = e_n(\vx) = 0 \) for \( n < 0 \). For a partition
\( \lambda=(\lambda_1,\dots,\lambda_\ell) \), we define
\( h_{\lambda}(\vx) = h_{\lambda_1}(\vx)\cdots h_{\lambda_\ell}(\vx)
\) and
\( e_{\lambda}(\vx) = e_{\lambda_1}(\vx)\cdots e_{\lambda_\ell}(\vx)
\).

\medskip

Plethystic substitution is an important notion in symmetric function
theory, which will play a key role in this paper. The plethystic
substitution \( f[Z] \) is defined for any symmetric function
\( f(\vx) \) and a formal power series \( Z \) in any variables.
However, in this paper, we only need the special cases that
\( f(\vx) \) is either \( h_n(\vx) \) or \( e_n(\vx) \), and \( Z \) is
a sum of variables with coefficients \( 1 \) or \( -1 \). Hence one
may take the following proposition as the definition of plethystic
substitution. We refer the reader to the first part
\cite{HJKSS2024} for the precise definition. See also
\cite{Loehr_2011, Macdonald} for more details on plethystic
substitution.

\begin{prop}
  Let \( Y = y_1 + \cdots + y_r \) and \( Z = z_1 + \cdots +z_s \),
  where \( y_1,\dots, y_r, z_1,\dots , z_s \) are formal
  variables. Then 
  \begin{align*}
    h_n[Y] &= h_n(y_1,\dots,y_r), & e_n[Y] &= e_n(y_1,\dots,y_r),\\
    h_n[-Y] &= (-1)^n e_n[Y], & e_n[-Y] &= (-1)^n h_n[Y],\\
   h_n[Y+Z] &= \sum_{a+b=n} h_a[Y] h_b[Z], & e_n[Y+Z] &= \sum_{a+b=n} e_a[Y] e_b[Z],\\
   h_n[Y-Z] &= \sum_{a+b=n} h_a[Y] h_b[-Z], & e_n[Y-Z] &= \sum_{a+b=n} e_a[Y] e_b[-Z].
  \end{align*}
  Here, it is possible that \( r=\infty \) or \( s=\infty \). If
  \( r=\infty \), then it means that \( Y = y_1 + y_2 + \cdots \) and
  \( h_n[Y] = h_n(y_1,y_2,\dots) \).
\end{prop}

We recall the notation \( \ominus \) introduced in the first part
\cite{HJKSS2024}:
\[
  h_n[Y\ominus Z] = \sum_{a-b=n} h_a[Y] h_b[Z], \qquad
  e_n[Y\ominus Z] = \sum_{a-b=n} e_a[Y] e_b[Z].
\]
For integers \( r,s \), and \( n \), let
\[
  X=x_1+x_2+\cdots, \qquad
  X_n=x_1+x_2+\dots+x_n, \qquad
  X_{[r,s]}=x_r+x_{r+1}+\dots+x_s,
\]
\[
  A_n = \alpha_1+\alpha_2+\dots+\alpha_n,\qquad
  B_n = \beta_1+\beta_2+\dots+\beta_n,
\]
where \( X_n=A_n=B_n=0 \) if \( n\le 0 \) and \( X_{[r,s]}=0 \) if \( r>s \).

The next four basic lemmas will be used later.

\begin{lem}[{\cite[Lemma~5.2]{Kim_JT22}}] \label{lem:h_m[Z]}
  Suppose that \( Z \) is any formal power series, \( z \) is a single
  variable, and \( c\in\{0,1\} \). Then we have
  \[
    h_m[Z] = h_m[Z-cz] + czh_{m-1}[Z].
  \]
\end{lem}

\begin{lem}[{\cite[Lemma~4.7]{Kim_JT22}}]  \label{lem:e_m[Z]}
  Suppose that \( Z \) is any formal power series, \( z \) is a single
  variable, and \( c\in\{0,1\} \). Then we have
  \[
    e_m[Z] = e_m[Z-cz] + cze_{m-1}[Z-cz] \qand e_m[Z] = e_m[Z+cz]-cze_{m-1}[Z].
  \]
\end{lem}

\begin{lem}[{\cite[Lemma~2.5]{HJKSS2024}}]\label{lem:det(h)=0}
  Let \( \lambda,\mu\in\Par_n \) with \( \mu\not\subseteq\lambda \)
  and let \( Z_{i,j} \) be any formal power series for \( 1\le i,j\le n \).
  Then
  \[
    \det \left( h_{\lambda_i-\mu_j-i+j}[Z_{i,j}] \right)_{i,j=1}^{n}
    =\det \left( e_{\lambda_i-\mu_j-i+j}[Z_{i,j}] \right)_{i,j=1}^{n} = 0. 
  \]
\end{lem}

\begin{lem}[{\cite[Lemma~2.6]{HJKSS2024}}]\label{lem:det(h)=det*det}
  Let \( \lambda,\mu\in\Par_n \) with \( \mu\subseteq\lambda \)
  and let \( Z_{i,j} \) be any formal power series for \( 1\le i,j\le n \).
  If \( 1\le k\le n \) is an integer such that \( \lambda_k=\mu_k \), then
  \[
    \det \left( h_{\lambda_i-\mu_j-i+j}[Z_{i,j}] \right)_{i,j=1}^{n}
    =\det \left( h_{\lambda_i-\mu_j-i+j}[Z_{i,j}] \right)_{i,j=1}^{k-1}
    \det \left( h_{\lambda_i-\mu_j-i+j}[Z_{i,j}] \right)_{i,j=k+1}^{n}. 
  \]
\end{lem}

We finish this section with the following definition, which will be
used frequently in this paper.

\begin{defn}
  A \emph{reverse partition} is a weakly increasing sequence of positive integers
\( \vs=(s_1,\dots,s_n) \). Denote by \( \RPar_n \) the set of reverse partitions
with \( n \) elements.
For convenience in our arguments, we define \( s_0=0 \) for \(
\vs=(s_1,\dots,s_n)\in\RPar_n \).
For \( \vr=(r_1,\dots,r_n),\vs=(s_1,\dots,s_n)\in\PP^n \),
we write \( \vr\le \vs \) if \( r_i\le s_i \) for all \( 1\le i\le n \).
\end{defn}

\section{Flagged Grothendieck polynomials}
\label{sec:flagged G}

In this section we give a combinatorial model for the refined canonical
stable Grothendieck polynomials \( G_\lambda(\vx;\va,\vb) \) using marked
multiset-valued tableaux. To this end we introduce a flagged version of \(
G_\lambda(\vx;\va,\vb) \) using marked multiset-valued tableaux and prove a
Jacobi--Trudi-like formula for this flagged version, which reduces to the
Jacobi--Trudi-like formula for \( G_\lambda(\vx;\va,\vb) \) in
Theorem~\ref{thm:JT_ab_intro}. More generally, we consider two flagged versions
of \( G_\lambda(\vx;\va,\vb) \) and extend the partition \( \lambda \) to a skew
shape. In the last part of this section we show that one of our results implies
a result of Matsumura~\cite{Matsumura_2018} on a Jacobi--Trudi-like formula
for a flagged Grothendieck polynomial.

To begin with, we introduce the following definition:

\begin{defn}
A \emph{marked multiset-valued tableau} of shape $\lm$ is a filling $T$ of $\lm$
with multisets such that
\begin{itemize}
\item $T(i,j)$ is a nonempty (finite) multiset $\{a_1\le \dots\le a_k\}$ of positive integers, in
  which each integer $a_i$ may be marked if $i\ge2$ and $a_{i-1}<a_i$, and
\item $\max(T(i,j)) \le \min(T(i,j+1))$ and $\max(T(i,j)) < \min(T(i+1,j))$
  if \( (i,j), (i,j+1)\in \lm \) and \( (i,j), (i+1,j)\in \lm \), respectively.
\end{itemize}
Let $\MMSVT(\lm)$ denote the set of marked multiset-valued tableaux of shape
$\lm$.

For $T\in\MMSVT(\lm)$, let \( \vx^{T(i,j)}=x_1^{m_1}x_2^{m_2}\cdots\), where \(
m_k \) is the total number of appearances of (unmarked integers) \( k \) and
(marked integers) \( k^* \) in \( T(i,j) \), and let \( \unmarked(T(i,j)) \)
(resp.~\( \marked(T(i,j)) \)) denote the number of unmarked (resp.~marked)
integers in \( T(i,j) \). We define
\[
  \wt(T) =  \prod_{(i,j)\in\lm} \vx^{T(i,j)} \alpha_j^{\unmarked(T(i,j))-1}
  (-\beta_i)^{\marked(T(i,j))}.
\]
For example, see Figure~\ref{fig:mmsvt}.
\end{defn}

\begin{figure}
\ytableausetup{boxsize=3.3em}
\begin{ytableau}
  \none & 1,2^*,2 & 2,2,4^* \\
  1 & 3,3
\end{ytableau}
\caption{An example of \( T\in\MMSVT((3,2)/(1)) \),
  where the marked integers are indicated with \( * \).  The weight of \( T \) is given by 
  \( \wt(T)=x_1^2x_2^4x_3^2x_4\alpha_2^2\alpha_3(-\beta_1)^2 \).}
\label{fig:mmsvt}
\end{figure}

The main goal of this section is to prove the following combinatorial
interpretation for \( G_\lambda(\vx;\va,\vb) \).

\begin{thm}\label{thm:main_G_sec6}
  For a partition \( \lambda \), we have
\[
  G_\lambda(\vx;\va,\vb) = \sum_{T\in \MMSVT(\lambda)} \wt(T).
\]
\end{thm}

\begin{remark}
  One can easily check that the marked multiset-valued tableaux are essentially
  the same as the hook-valued tableaux introduced in \cite{Yeliussizov2017}. If \(
  \va=(\alpha,\alpha,\dots)\) and \( \vb=(-\beta,-\beta,\dots) \),
  Theorem~\ref{thm:main_G_sec6} reduces to Yeliussizov's result in \cite[Theorem
  4.2]{Yeliussizov2017}.
\end{remark}

\begin{remark}
  We note that the crystal operators introduced by Hawkes and Scrimshaw
  \cite[Theorem 4.6]{HS20} on MMSVT (when translated from hook-valued tableaux
  under the natural bijection) preserve the statistics that determine \( \va \)
  and \( \vb \), which gives an alternative proof of the Schur positivity of \(
  G_\lambda(\vx;\va,-\vb) \) (private communication with Travis Scrimshaw).
\end{remark}

Our strategy for proving Theorem~\ref{thm:main_G_sec6} is to introduce a flagged version which
allows us to use induction. We note that a row-flagged version with partition
shape is sufficient to prove this theorem, but for completeness we also consider a
column-flagged version and we generalize the partition shape to any skew shape.

\begin{defn}\label{defn:Grow}
  Let \( \vr=(r_1,\dots,r_n)\in\NN^n \) and \( \vs=(s_1,\dots,s_n)\in\NN^n \).
  Let \(\MMSVT^{\row(\vr, \vs)}(\lm)\) denote the set of $T\in \MMSVT(\lm)$ such
  that $r_i \le \min(T(i,j))$ and $\max(T(i,j))\le s_i$ for all $(i,j)\in \lm$.
  Similarly, $\MMSVT^{\col(\vr, \vs)}(\lm)$ denotes the set of $T\in \MMSVT(\lm)$
  such that $r_j \le \min(T(i,j))$ and $\max(T(i,j))\le s_j$ for all $(i,j)\in
  \lm$. We define the \emph{row-flagged} and \emph{column-flagged
    refined canonical stable Grothendieck polynomials} by
  \begin{align*}
    G_\lm^{\row(\vr, \vs)}(\vx;\va,\vb)
    &=\sum_{T\in\MMSVT^{\row(\vr, \vs)}(\lm)} \wt(T),\\
    G_\lm^{\col(\vr, \vs)}(\vx;\va,\vb)
    &=\sum_{T\in\MMSVT^{\col(\vr, \vs)}(\lm)} \wt(T).
  \end{align*}
\end{defn}

\begin{remark}\label{rem:s_k=1}
  Recall that if \( T\in \MMSVT(\lm) \), then every cell of \( T \)
  consists only of positive integers. Therefore, in
  Definition~\ref{defn:Grow}, it suffices to consider
  \( \vr,\vs\in\PP^n \). However, we extend the definition to
  \( \vr,\vs\in\NN^n \) because later in this section, for example in
  Lemma~\ref{lem:rec2}, we will prove recurrences for
  \( G_\lm^{\row(\vr, \vs)}(\vx;\va,\vb) \) in which an element
  \( s_k \) of \( \vs \) is decreased by \( 1 \). 

  To illustrate, suppose \( s_k=0 \). If \( \mu_k<\lambda_k \), then
  there is a cell in row \( k \), say \( (i,k) \), in
  \( \lambda/\mu \). Thus, if
  \( T\in \MMSVT^{\row(\vr, \vs)}(\lm) \), then every integer
  in \( T(i,k) \) must be at most \( s_k=0 \). However, since
  \( T(i,k) \) must consist of positive integers, there is no such
  \( T \), hence
  \( \MMSVT^{\row(\vr, \vs)}(\lm) = \emptyset \). On the other
  hand, if \( \mu_k=\lambda_k \), then there are no cells in row
  \( k \) of \( \lambda/\mu \). Therefore the upper bound \( s_k=0 \)
  is irrelevant and we may have
  \( \MMSVT^{\row(\vr, \vs)}(\lm) \ne \emptyset \).
\end{remark}

Now we state Jacobi--Trudi-like formulas for \( G_\lm^{\row(\vr,
  \vs)}(\vx;\va,\vb) \) and \( G_\lm^{\col(\vr, \vs)}(\vx;\va,\vb) \) whose
proofs are given in Sections~\ref{sec:proof-theor-refthm:m-1} and
\ref{sec:proof-theor-refthm:m-2} respectively.
Here we use the convention that
if \( r>s \), then the product \( \prod_{i=r}^{s} a_i \) is
defined to be \( 1 \).

\begin{thm}
  \label{thm:main_G_row}
  Let \( \lambda,\mu\in\Par_n,\vr,\vs\in\PP^n \) with \( \mu\subseteq\lambda \).
  If $r_i\le r_{i+1}$ and $s_i\le s_{i+1}$ whenever
  $\mu_i<\lambda_{i+1}$ for \( 1\le i\le n-1 \), then
  \begin{align}
    \label{eq:G_row}
    G_\lm^{\row(\vr, \vs)}(\vx;\va,\vb) 
    &= C \det\left(   h_{\lambda_i-\mu_j-i+j}
      [X_{[r_j,s_i]}\ominus(A_{\lambda_i}-A_{\mu_j}-B_{i-1} + B_{j})]\right)_{i,j=1}^n,
  \end{align}
where
\[
C = \prod_{i=1}^n \prod_{l=r_i}^{s_i} (1-\beta_i x_l).
\]
\end{thm}

\begin{thm}
  \label{thm:main_G_col}
  Let \( \lambda,\mu\in\Par_n,\vr,\vs\in\PP^n \) with \( \mu\subseteq\lambda \).
If \( r_i-\mu_i\le r_{i+1}-\mu_{i+1} \) and \( s_i-\lambda_i\le s_{i+1}-\lambda_{i+1}+1 \) 
whenever \( \mu_i<\lambda_{i+1} \) for \( 1\le i\le n-1 \), then  
  \begin{align} \label{eq:G_col}
    G_\lmc^{\col(\vr, \vs)}(\vx;\va,\vb) 
    = D \det\left(   e_{\lambda_i-\mu_j-i+j}
      [X_{[r_j,s_i]}\ominus(A_{i-1}-A_{j}-B_{\lambda_i} + B_{\mu_j})]\right)_{i,j=1}^n,
  \end{align}
where
\[
D = \prod_{i=1}^n \prod_{l=r_i}^{s_i} (1-\alpha_i x_l)^{-1}.
\]
\end{thm}

\begin{remark}
  The condition \( \mu\subseteq\lambda \) in Theorems~\ref{thm:main_G_row} and \ref{thm:main_G_col}
  is necessary. Indeed, if \( \lambda=(1),\mu=(2),\vr=(1), \vs=(1)\), 
  then the left hand side of \eqref{eq:G_row} (resp.~\eqref{eq:G_col}) is zero, while the right hand side is 
  \( \beta_1-\alpha_2 \) (resp.~\( \beta_2-\alpha_1 \)).
\end{remark} 

By specializing Theorem~\ref{thm:main_G_row} we obtain the following corollary,
which implies the combinatorial model for \( G_\lambda(\vx;\va,\vb) \) in
Theorem~\ref{thm:main_G_sec6} as \( n\to\infty \).

\begin{cor} \label{cor:G=MMSVT}
  For a partition \( \lambda\in\Par_n \),
  we have
  \[
    G_\lambda(\vx_n;\va,\vb) = \sum_{T\in \MMSVT(\lambda), \max(T)\le n} \wt(T),
  \]
  where \( \max(T) \) is the largest integer appearing in \( T \).
\end{cor}
\begin{proof}
  Let \( \vr=(1^n) \) and \( \vs=(n^n) \), where \( (a^n) \) means the sequence
  \((a,a,\dots,a) \) with \( n \) \( a \)'s. By definition, we have
  \[
    G_\lambda^{\row(\vr, \vs)}(\vx;\va,\vb)
    = \sum_{T\in\MMSVT(\lambda), \max(T)\le n} \wt(T).
  \]
  On the other hand, by Theorems~\ref{thm:main_G_row} and \ref{thm:JT_ab_intro}, we have
  \begin{align*}
    G_\lambda^{\row(\vr, \vs)}(\vx;\va,\vb) 
    &= \prod_{i,j=1}^n  (1-\beta_i x_j) \det\left(   h_{\lambda_i-i+j}
      [X_{[r_j,s_i]}\ominus(A_{\lambda_i}-B_{i-1} + B_{j})]\right)_{i,j=1}^n\\
    &= G_\lambda(\vx_n;\va,\vb).
  \end{align*}
Combining the above two equations gives the corollary.
\end{proof}

\begin{remark}
  Note that the Jacobi--Trudi formula for the Schur function
  \( s_{\lm}(\vx) \) of shape \(\lm\) does not depend on the
  choice of \(n\) for \(n> \ell(\lambda)\), that is,
  \begin{align*}
    s_\lm(\vx)
    =\det \left( h_{\lambda_i-\mu_j-i+j}(\vx) \right)_{i,j=1}^n
    =\det \left( h_{\lambda_i-\mu_j-i+j}(\vx) \right)_{i,j=1}^{\ell(\lambda)}.
  \end{align*}
  This can be easily proved using the fact that \(h_m(\vx)=0\) for \(m<0\).
  However, since \(h_m[X\ominus Y]\not= 0\) for \(m<0\) in general, it is not
  obvious that the Jacobi--Trudi-like formula for \(G_\lm^{\row(\vr,
    \vs)}(\vx;\va,\vb)\) in Theorem~\ref{thm:main_G_row} does not depend on the
  choice of \(n\) for \(n>\ell(\lambda)\). It is still possible to prove the
  independence of \( n \) directly by showing that
  for \( i \)  and \( j \) satisfying \( j<i \) and \( \lambda_i=\mu_i=0 \),
   the \( (i,j) \)-entry of the matrix in \eqref{eq:G_row} is zero.
  \end{remark}

In the next two subsections we prove Theorems~\ref{thm:main_G_row} and \ref{thm:main_G_col}.

\subsection{Proof of Theorem~\ref{thm:main_G_row}: row-flagged Grothendieck polynomials} 
\label{sec:proof-theor-refthm:m-1}

For \( \lambda,\mu\in\Par_n, \) and \( \vr,\vs\in\NN^n \), we define
\begin{equation}\label{eq:tG}
  \tG_{\lm}^{\row(\vr,\vs)}(\vx;\va,\vb)=C \det\left(   h_{\lambda_i-\mu_j-i+j}
    [X_{[r_j,s_i]}\ominus(A_{\lambda_i}-A_{\mu_j}-B_{i-1} + B_{j})]\right)_{i,j=1}^n,
\end{equation}
where \( C=\prod_{i=1}^n \prod_{l=r_i}^{s_i} (1-\beta_i x_l) \).
We prove the Jacobi--Trudi-like identity~\eqref{eq:G_row} in Theorem~\ref{thm:main_G_row} by showing that 
\( G_{\lm}^{\row(\vr,\vs)}(\vx;\va,\vb)$ and $\tG_{\lm}^{\row(\vr,\vs)}(\vx;\va,\vb) \) 
satisfy the same recurrence relations and initial conditions.
 The idea is based on a proof of the Jacobi--Trudi-like 
 formula for flagged Schur functions due to Wachs \cite{Wachs_1985}.

 Throughout this section we assume that \( n \) is fixed and $\epsilon_k$
 denotes the 0-1 sequence of length \( n \) whose $k$th entry is \( 1 \) and all
 other entries are \( 0 \).

 In order to prove Theorem~\ref{thm:main_G_row} we need several lemmas. Let \(
 \phi_{\vb} \) be the substitution homomorphism that replaces \( \beta_{i} \)
 by \( \beta_{i+1} \) for each \( i\ge 1 \). For example, \( \phi_{\vb}(x_1
 \beta_1\beta_3 + x_2 \alpha_2 \beta_1 + \alpha_3) = x_1 \beta_2\beta_4 + x_2 \alpha_2 \beta_2
 + \alpha_3 \).

 \begin{lem}\label{lem:rec1}
   Let \( \lambda,\mu\in\Par_n,\vr,\vs\in\NN^n \)
  with \( \mu\subseteq\lambda \) and let \( 1\le k\le n-1 \) be an integer such that
  $\mu_k\ge\lambda_{k+1}$. Then
\begin{align*}
  G_{\lm}^{\row(\vr,\vs)}(\vx;\va,\vb)
  &=G_{{\lambda}^{(1)}/{\mu}^{(1)}}^{\row({\vr}^{(1)},{\vs}^{(1)})}(\vx;\va,\vb)\cdot\phi_{\vb}^k
    \left(G_{{\lambda}^{(2)}/{\mu}^{(2)}}^{\row({\vr}^{(2)},{\vs}^{(2)})}(\vx;\va,\vb)\right),\\
  \tG_{\lambda/\mu}^{\row(\vr,\vs)}(\vx;\va,\vb)
  &=\tG_{{\lambda}^{(1)}/{\mu}^{(1)}}^{\row({\vr}^{(1)},{\vs}^{(1)})}(\vx;\va,\vb)\cdot\phi_{\vb}^k
    \left(\tG_{{\lambda}^{(2)}/{\mu}^{(2)}}^{\row({\vr}^{(2)},{\vs}^{(2)})}(\vx;\va,\vb)\right),
\end{align*}
where \( \gamma^{(1)} = (\gamma_1,\dots,\gamma_k) \) and
    \( \gamma^{(2)} = (\gamma_{k+1},\dots,\gamma_n) \)
    for each \( \gamma\in \{\lambda,\mu,\vr,\vs\} \).
\end{lem}
\begin{proof}
  Observe that \( \lm \) is the union of \( {\lambda}^{(1)}/{\mu}^{(1)} \) and
  \( {\lambda}^{(2)}/{\mu}^{(2)} \) whose column indices are disjoint. Thus, the
  first identity directly follows from the definition of
  $G_{\lambda/\mu}^{\row(\vr,\vs)}(\vx;\va,\vb)$.
  
  For the second identity, it is easy to check that both sides have the same factor
  \( \prod_{i=1}^n \prod_{l=r_i}^{s_i} (1-\beta_i x_l). \)
  We will show that the matrix in
  \eqref{eq:tG} is a block upper triangular matrix with diagonal blocks of sizes
  \( k\times k \) and \( (n-k)\times (n-k) \). Since
  $\mu_1\ge\cdots\ge\mu_k\ge\lambda_{k+1}\ge\cdots\ge\lambda_{n}$,
  we have $\lambda_i-\mu_j-i+j<0$ for all \( k+1\le i\le n \) and $1\le j\le k$.
  Furthermore, for such \( i \) and \( j \), the \( (i,j) \)-entry of the matrix
  in \eqref{eq:tG} is equal to
\begin{align*}
  &\sum_{\ell\ge0 }h_{\lambda _{i}-\mu _{j}-i+j+\ell }\left[
  X_{[r_{j},s_{i}]}\right] h_{\ell }\left[ A_{\lambda _{i}}-A_{\mu
  _{j}}-B_{i-1}+B_{j}\right]\\
  &=\sum_{\ell\ge-\lambda_{i}+\mu_{j}+i-j}h_{\lambda _{i}-\mu _{j}-i+j+\ell }\left[
    X_{[r_{j},s_{i}]}\right] h_{\ell }\left[ A_{\lambda _{i}}-A_{\mu
    _{j}}-B_{i-1}+B_{j}\right]\\ 
  &=\sum_{\ell\ge-\lambda_{i}+\mu_{j}+i-j}h_{\lambda _{i}-\mu _{j}-i+j+\ell }\left[
    X_{[r_{j},s_{i}]}\right] (-1)^\ell e_{\ell }\left(\alpha_{\lambda_{i}+1},\dots,\alpha_{\mu_j},\beta_{j+1},\dots,\beta_{i-1}\right),
\end{align*}
which is equal to \( 0 \)
because for each \( \ell\ge-\lambda_{i}+\mu_{j}+i-j \), the number of variables in
\( e_\ell(\alpha_{\lambda_{i}+1},\dots,\alpha_{\mu_j},\beta_{j+1},\dots,\beta_{i-1}) \) is \( (\mu_j-\lambda_i)+(i-1-j)<\ell \). Thus the matrix in
\eqref{eq:tG} is a block upper triangular matrix and we obtain the second
identity.
\end{proof}

The following two lemmas show that
\( G_{\lm}^{\row(\vr,\vs)}(\vx;\va,\vb) \) and \( \tG_{\lm}^{\row(\vr,\vs)}(\vx;\va,\vb) \)
have the same initial conditions.

\begin{lem}\label{lem:G=1}
  For \( \lambda\in\Par_n,\vr,\vs\in\NN^n \),
  \[
    G_{\lambda/\lambda}^{\row(\vr,\vs)}(\vx;\va,\vb) = 
    \tG_{\lambda/\lambda}^{\row(\vr,\vs)}(\vx;\va,\vb) = 1.
  \]
\end{lem}
\begin{proof}
  Since \(\MMSVT^{\row(\vr, \vs)}(\lambda/\lambda)\) has only one element,
  namely the empty tableau, we have
  \( G_{\lambda/\lambda}^{\row(\vr,\vs)} (\vx;\va,\vb)=1\).

  For the second equality, let  \( H_{i,j} = h_{\lambda_i-\lambda_j-i+j}
  [X_{[r_j,s_i]}\ominus(A_{\lambda_i}-A_{\lambda_j}-B_{i-1} + B_{j})]\) so that 
  \(\tG_{\lambda/\lambda}^{\row(\vr,\vs)}(\vx;\va,\vb)=C\det(H_{i,j})_{i,j=1}^{n}\),
  where \(C=\prod_{i=1}^{n}\prod_{l=r_i}^{s_i} (1-\beta_i x_l) \).
  If we set \( \mu=\lambda \), then since
  \( \mu_k=\lambda_k\ge\lambda_{k+1} \) for all \( 1\le k\le n-1 \), we have
  \(C\det(H_{i,j})_{i,j=1}^{n}=C\prod_{i=1}^{n}H_{i,i} \) by
  iterated application of Lemma~\ref{lem:rec1}. Since 
  \begin{equation}\label{eq:2} 
    H_{i,i}=\sum_{\ell\ge0}h_\ell[X_{[r_i,s_i]}]h_\ell[\beta_i]
    =\sum_{\ell\ge0}h_\ell[X_{[r_i,s_i]}]\beta_i^\ell
    =\prod_{l=r_i}^{s_i}(1-\beta_ix_l)^{-1},
  \end{equation}
  we obtain \( \tG_{\lambda/\lambda}^{\row(\vr,\vs)}(\vx;\va,\vb)=C\det(H_{i,j})_{i,j=1}^{n}=C\cdot C^{-1}=1 \).
  Note that if \( r_i>s_i \), then \( X_{[r_i,s_i]}=0 \), and therefore
  \eqref{eq:2} is still true since
  \( \sum_{\ell\ge0}h_\ell[X_{[r_i,s_i]}]\beta_i^\ell=1
=\prod_{l=r_i}^{s_i}(1-\beta_ix_l)^{-1} \).
\end{proof}

\begin{lem}\label{lem:rec4}
Let \( \lambda,\mu\in\Par_n,\vr,\vs\in\RPar_n \) with \( \mu\subseteq\lambda \).
If \( s_k<r_k \) and \( \mu_k<\lambda_k \) for some \( 1\le k\le n \), then
\[
  G_{\lambda/\mu}^{\row(\vr,\vs)}(\vx;\va,\vb)= \tG_{\lambda/\mu}^{\row(\vr,\vs)}(\vx;\va,\vb)=0.
\]
\end{lem}
\begin{proof}
  Since \(\MMSVT^{\row(\vr, \vs)}(\lambda/\mu)=\emptyset\),
  we have \( G_{\lambda/\mu}^{\row(\vr,\vs)}(\vx;\va,\vb)=0 \).   
  To prove the second equality, we claim that for all $1\le i\le k$ and \( k\le j\le n \),
    the \( (i,j) \)-entry of the matrix in \eqref{eq:tG} is \( 0, \) that is,
    \[
        \sum_{\ell\ge0 }h_{\lambda _{i}-\mu _{j}-i+j+\ell }\left[
          X_{[r_{j},s_{i}]}\right] h_{\ell }\left[ A_{\lambda _{i}}-A_{\mu
            _{j}}-B_{i-1}+B_{j}\right]=0.
    \]
    Observe that $s_i\le s_k<r_k\le r_j$ and $\mu_j\le\mu_k<\lambda_k\le\lambda_i$
    by the assumptions of the lemma. Thus \(
    h_{\lambda_i-\mu_j-i+j+\ell}[X_{[r_j,s_i]}]=0\) for \( \ell\ge0\) since \(
    X_{[r_j,s_i]}=0 \) and \( \lambda_i-\mu_j-i+j+\ell>0 \). Therefore, the claim is
    true, which implies $\tG_{\lambda/\mu}^{\row(\vr,\vs)}(\vx;\va,\vb)=0$
    (because if an \( n\times n \)-matrix \( (a_{i,j}) \) satisfies \( a_{i,j}=0 \)
    for all \( 1\le i\le k \) and \( k\le j\le n \), then its determinant is 0).
\end{proof}

Recall that for \( \lambda\in \Par_n \) and \( \vs\in \RPar_n \) we assume \(
\lambda_{n+1}=0 \) and \( s_0=0 \). The following two lemmas show that \(
G_{\lm}^{\row(\vr,\vs)}(\vx;\va,\vb) \) and \(
\tG_{\lm}^{\row(\vr,\vs)}(\vx;\va,\vb) \) have the same recursions.

\begin{lem}\label{lem:rec2} 
  Let \( \lambda,\mu\in\Par_n,\vr,\vs\in\RPar_n \) with \( \mu\subseteq\lambda \). 
  Let \( 1\le k\le n \) be an integer satisfying the following conditions:
\begin{enumerate}
\item \( \mu_k<\lambda_k \),
\item \( r_k\le s_k \),
\item \( s_{k-1}<s_k \), and
\item \( \lambda_{k}>\lambda_{k+1} \).
\end{enumerate}
Then
\begin{align*}
  G_{\lambda/\mu}^{\row(\vr,\vs)}(\vx;\va,\vb)
  &=\frac{1-x_{s_k}\beta_k}{1-x_{s_k}\alpha_{\lambda_k}}G_{\lm}^{\row(\vr,\vs-\epsilon_k)}(\vx;\va,\vb)
    +\frac{x_{s_k}}{1-x_{s_k}\alpha_{\lambda_k}}G_{(\lambda-\epsilon_k)/\mu}^{\row(\vr,\vs)}(\vx;\va,\vb),\\
  \tG_{\lambda/\mu}^{\row(\vr,\vs)}(\vx;\va,\vb)
  &=\frac{1-x_{s_k}\beta_k}{1-x_{s_k}\alpha_{\lambda_k}}\tG_{\lm}^{\row(\vr,\vs-\epsilon_k)}(\vx;\va,\vb)
    +\frac{x_{s_k}}{1-x_{s_k}\alpha_{\lambda_k}}\tG_{(\lambda-\epsilon_k)/\mu}^{\row(\vr,\vs)}(\vx;\va,\vb).
\end{align*}
Moreover, if \( s_k=1 \), then
\[
 G_{\lm}^{\row(\vr,\vs-\epsilon_k)}(\vx;\va,\vb) 
 = \tG_{\lm}^{\row(\vr,\vs-\epsilon_k)}(\vx;\va,\vb) =0.
\]
\end{lem}

\begin{proof}
  Considering the elements in the cell \( (k,\lambda_k) \),
  we partition $\MMSVT^{\row(\vr,\vs)}(\lm)$ into two sets:
  \begin{align*}
    S_1 &=\left\{T\in \MMSVT^{\row(\vr,\vs)}(\lm): T(k,\lambda_k) \text{
    contains an integer }\ell<s_k\right\},\\
    S_2&=\left\{T\in \MMSVT^{\row(\vr,\vs)}(\lm): T(k,\lambda_k)\text{
    consists only of \( s_k \)'s} \right\}.
  \end{align*}
  For each $T\in S_1$, by extracting all $s_k$'s from $T(k,\lambda_k)$, we
  obtain
  \[
    \sum_{T\in
      S_1}
    \wt(T)=\frac{1-x_{s_k}\beta_k}{1-x_{s_k}\alpha_{\lambda_k}}G_{\lm}^{\row(\vr,\vs-\epsilon_k)}(\vx;\va,\vb).
  \]
  Note that if \( s_k=1 \), then \( S_1=\emptyset \) and, by
  Remark~\ref{rem:s_k=1},
  \( \MMSVT^{\row(\vr, \vs-\epsilon_k)}(\lm)=\emptyset \). Hence in
  this case both sides of the above identity are equal to \( 0 \)
  and in particular we have
  \( G_{\lm}^{\row(\vr,\vs-\epsilon_k)}(\vx;\va,\vb)=0 \). For each
  $T\in S_2$, extracting all $s_k$'s from $T(k,\lambda_k)$ and
  deleting the cell \( (k,\lambda_k) \), we obtain
  \[
    \sum_{T\in
      S_2}
    \wt(T)=\frac{x_{s_k}}{1-x_{s_k}\alpha_{\lambda_k}}G_{(\lambda-\epsilon_k)/\mu}^{\row(\vr,\vs)}(\vx;\va,\vb).
  \]
  Adding the above two equations gives the first identity.

For the second identity let
  \begin{align*}
    H_{i,j} &= h_{\lambda_i-\mu_j-i+j} [X_{[r_j,s_i]}\ominus(A_{\lambda_i}-A_{\mu_j}-B_{i-1} + B_{j})],\\
    F_{i,j} &= h_{\lambda_i-\mu_j-i+j} [X_{[r_j,s_i-\chi(i=k)]}\ominus(A_{\lambda_i}-A_{\mu_j}-B_{i-1} + B_{j})],\\
    G_{i,j} &= h_{\lambda_i-\mu_j-i+j-\chi(i=k)} [X_{[r_j,s_i]}\ominus(A_{\lambda_i-\chi(i=k)}-A_{\mu_j}-B_{i-1} + B_{j})].
  \end{align*}
By \eqref{eq:tG} the second identity is written as
  \[
    C \det(H_{i,j})  = \frac{C}{1-x_{s_k}\alpha_{\lambda_k}} \det(F_{i,j})
    + \frac{C x_{s_k}}{1-x_{s_k}\alpha_{\lambda_k}} \det(G_{i,j}),
  \]
  where \(C=\prod_{i=1}^{n}\prod_{l=r_i}^{s_i} (1-\beta_i x_l) \).
  Since \( H_{i,j}=F_{i,j}=G_{i,j} \) for all \( 1\le i,j\le n \) with \( i\ne k \), by the linearity of a
  determinant in its \( k \)th row, it suffices to show that
\begin{equation}\label{eq:HFG}
  H_{k,j}=\frac{1}{1-x_{s_k}\alpha_{\lambda_k}}F_{k,j}+\frac{x_{s_k}}{1-x_{s_k}\alpha_{\lambda_k}}G_{k,j}
\end{equation}
for all \( 1\le j\le n\). For brevity let \( m=\lambda _{k}-\mu _{j}-k+j \) and \( Z=A_{\lambda
  _{k}}-A_{\mu _{j}}-B_{k-1}+B_{j}\) so that
\begin{align}
  \label{eq:Hik}
  H_{k,j}&=\sum_{\ell\ge0 }h_{m+\ell }\left[
           X_{[r_{j},s_{k}]}\right] h_{\ell }[Z],  \\ 
  \label{eq:Fik}
  F_{k,j}&=\sum_{\ell\ge0 }h_{m+\ell }\left[
           X_{[r_{j},s_{k}-1]}\right] h_{\ell }[Z], \\
  \label{eq:Gik}
  G_{k,j}&=\sum_{\ell\ge0 }h_{m+\ell-1}\left[
           X_{[r_{j},s_{k}]}\right] h_{\ell }[Z-\alpha_{\lambda_k}].
\end{align}

We consider the two cases  \( r_j\le s_k\) and \( r_j> s_k\).
First, suppose \( r_j\le s_k\). We will express \( F_{k,j} \) in terms of
\( H_{k,j} \) and \( G_{k,j} \). Since \( h_{m+\ell}[X_{[r_{j},s_{k}-1]}] =
h_{m+\ell}[X_{[r_{j},s_{k}]}-x_{s_k}] \), by Lemma~\ref{lem:h_m[Z]} we have
\begin{align*}
  h_{m+\ell }\left[X_{[r_{j},s_{k}-1]}\right]
  &= h_{m+\ell }\left[X_{[r_{j},s_{k}]}\right] - x_{s_k}h_{m+\ell-1}\left[X_{[r_{j},s_{k}]}\right],\\
  h_{\ell}[Z]
  &= h_{\ell }[Z-\alpha_{\lambda_k}] +\alpha_{\lambda_k}h_{\ell-1}[Z].
\end{align*}
Thus
\begin{align*}
  F_{k,j}&=\sum_{\ell\ge0 }h_{m+\ell }\left[
    X_{[r_{j},s_{k}]}\right] h_{\ell }[Z]
  -x_{s_k} \sum_{\ell\ge0 }h_{m+\ell-1}\left[
    X_{[r_{j},s_{k}]}\right] h_{\ell }[Z]\\
         &= H_{k,j}-x_{s_k} \sum_{\ell\ge0 }h_{m+\ell-1}\left[
           X_{[r_{j},s_{k}]}\right] \left(h_{\ell }[Z-\alpha_{\lambda_k}] +\alpha_{\lambda_k}h_{\ell-1}[Z]\right)\\
         &= H_{k,j}-x_{s_k} G_{k,j} -x_{s_k}\alpha_{\lambda_k} H_{k,j},
\end{align*}
which is equivalent to \eqref{eq:HFG}.

Now suppose \( r_j>s_k\). We will show that \( H_{k,j} = F_{k,j} = G_{k,j}=0 \),
which implies \eqref{eq:HFG}. First observe that, since \(
X_{[r_j,s_k]}=X_{[r_j,s_k-1]}=0 \) and \( h_t[0]=\chi(t=0) \), we obtain from
\eqref{eq:Hik}, \eqref{eq:Fik}, and \eqref{eq:Gik} that
\begin{equation}\label{eq:h-m}
  H_{k,j} = F_{k,j} = h_{-m}[Z],  \qquad G_{k,j} = h_{1-m}[Z-\alpha_{\lambda_k}], 
\end{equation}
where \( m=\lambda _{k}-\mu _{j}-k+j \) and \( Z=A_{\lambda _{k}}-A_{\mu
  _{j}}-B_{k-1}+B_{j}\) as before and \( h_t(\vx)=0 \) for \( t<0 \). Since
\( r_k\le s_k<r_j \) and \( \vr\in\RPar_n \), we must have \( j>k \). Thus \(
\lambda_k>\mu_k\ge\mu_j \) and hence \( m=(\lambda_k-\mu_j)+(j-k)\ge1+1\), which
together with \eqref{eq:h-m} implies \( H_{k,j}=F_{k,j}=G_{k,j}=0\). Thus, \eqref{eq:HFG} holds. 

To complete the proof, it remains to show that
\( \tG_{\lambda/\mu}^{\row(\vr,\vs-\epsilon_k)}(\vx;\va,\vb)=0 \) for the case
\( s_k=1 \). If \( s_k=1 \), then \( r_j>s_k-1=0 \). Applying the same
argument as before, we have \( F_{k,j}=0 \) for any \( j \) and hence
\( \tG_{\lambda/\mu}^{\row(\vr,\vs-\epsilon_k)}(\vx;\va,\vb)=0. \)
\end{proof}

\begin{lem}\label{lem:rec3}
  Let \( \lambda,\mu\in\Par_n,\vr,\vs\in\RPar_n \) with \( \mu\subseteq\lambda
  \). Let \( 1\le k\le n-1 \) be an integer satisfying the following conditions:
\begin{enumerate}
\item \( \mu_k<\lambda_k \),
\item \( r_k\le s_k \),
\item \( s_k=s_{k+1} \), and
\item \( \lambda_k=\lambda_{k+1} \).
\end{enumerate}
Then
\begin{align*}
G_{\lambda/\mu}^{\row(\vr,\vs)}(\vx;\va,\vb)&=G_{\lambda/\mu}^{\row(\vr,\vs-\epsilon_k)}(\vx;\va,\vb),\\
\tG_{\lambda/\mu}^{\row(\vr,\vs)}(\vx;\va,\vb)&=\tG_{\lambda/\mu}^{\row(\vr,\vs-\epsilon_k)}(\vx;\va,\vb).
\end{align*}
Moreover, if \( s_k=1 \), then
\[
 G_{\lm}^{\row(\vr,\vs-\epsilon_k)}(\vx;\va,\vb) 
 = \tG_{\lm}^{\row(\vr,\vs-\epsilon_k)}(\vx;\va,\vb) =0.
\]
\end{lem}
\begin{proof}
  
  Let \( T\in \MMSVT^{\row(\vr,\vs)}(\lm) \). Then 
  \[
    \max(T(k,\lambda_k))<\min(T(k+1,\lambda_k)) = \min(T(k+1,\lambda_{k+1})) \le
    s_{k+1} = s_k.
  \]
  Thus every integer in the \( k \)th row of \( T \) is at most
  \( s_k-1 \), which shows the first identity. As before, we have
  \( G_{\lm}^{\row(\vr,\vs-\epsilon_k)}(\vx;\va,\vb)=0 \) if
  \( s_k=1 \) by Remark~\ref{rem:s_k=1}.

  For the second identity let
  \begin{align*}
    H_{i,j} &= h_{\lambda_i-\mu_j-i+j} [X_{[r_j,s_i]}\ominus(A_{\lambda_i}-A_{\mu_j}-B_{i-1} + B_{j})],\\
    F_{i,j} &= h_{\lambda_i-\mu_j-i+j} [X_{[r_j,s_i-\chi(i=k)]}\ominus(A_{\lambda_i}-A_{\mu_j}-B_{i-1} + B_{j})].
  \end{align*}
 Dividing both sides by \( \prod_{i=1}^{n}\prod_{l=r_i}^{s_i}
  (1-\beta_i x_l) \), the second identity can be rewritten as
  \begin{equation}\label{eq:H=F}
\det(H_{i,j}) = \frac{1}{1-x_{s_k}\beta_k}\det(F_{i,j}).
  \end{equation}

  We claim that for all \( 1\le j\le n \),
  \begin{equation}\label{eq:F=H-H}
    F_{k,j}=(1-x_{s_k}\beta_k)H_{k,j}-x_{s_{k}}H_{k+1,j}.
  \end{equation}
  Since \( F_{i,j}=H_{i,j} \) for all \( 1\le i,j\le n \) with \( i\ne k \), by
  the linearity of a determinant in its \( k \)th row, the claim implies \eqref{eq:H=F}. To prove the claim, 
  we compare \( F_{k,j} \) and \( H_{k,j} \).
  Let \( m=\lambda_{k}-\mu _{j}-k+j \) and \( Z=A_{\lambda
    _{k}}-A_{\mu _{j}}-B_{k-1}+B_{j}\) for brevity.
  We consider the two cases \( r_j\le s_k\) and \( r_j> s_k\). 
  
  First, suppose \( r_j\le s_k\). By Lemma~\ref{lem:h_m[Z]}, we have
  \begin{align*}
    h_{m+\ell }\left[X_{[r_{j},s_{k}-1]}\right]
    &= h_{m+\ell }\left[X_{[r_{j},s_{k}]}\right] - x_{s_k}h_{m+\ell-1}\left[X_{[r_{j},s_{k}]}\right],\\
    h_{\ell}[Z]
    &= h_{\ell }[Z-\beta_{k}] +\beta_{k}h_{\ell-1}[Z].
  \end{align*}
  The claim then follows from similar computations as in the proof of
  Lemma~\ref{lem:rec2} with the above two equations
  since \( s_k=s_{k+1} \) and \( \lambda_k=\lambda_{k+1} \) easily lead to
  \( H_{k+1,j}=\sum_{\ell\ge 0} h_{m+\ell-1} [X_{[r_j, s_k]}] h_\ell[Z-\beta_k] \).

  Now suppose \( r_j>s_k\). By the same argument as in the proof of
  Lemma~\ref{lem:rec2}, we obtain \( H_{k,j}=F_{k,j}=0\) and also \(
  H_{k+1,j}=0\). This shows the claim \eqref{eq:F=H-H}.

  To complete the proof, it remains to show that
  \( \tG_{\lambda/\mu}^{\row(\vr,\vs-\epsilon_k)}(\vx;\va,\vb)=0 \) for the case
  \( s_k=1 \). In this case, since \( r_j>s_k-1=0 \), one can apply the same
  argument above to obtain \( F_{k,j}=0 \) for any \( j \), which gives
  \( \tG_{\lambda/\mu}^{\row(\vr,\vs-\epsilon_k)}(\vx;\va,\vb)=0 \).
\end{proof}

We are now ready to prove Theorem~\ref{thm:main_G_row}, which is restated as the
following theorem.

\begin{thm}\label{thm:row_f_G}
  Let \( \lambda,\mu\in\Par_n,\vr,\vs\in\PP^n \) with \( \mu\subseteq\lambda \).
If $r_i\le r_{i+1}$ and $s_i\le s_{i+1}$ whenever
$\mu_i<\lambda_{i+1}$ for \( 1\le i\le n-1 \), then
\[ 
  G_{\lambda/\mu}^{\row(\vr,\vs)}(\vx;\va,\vb)=\tG_{\lambda/\mu}^{\row(\vr,\vs)}(\vx;\va,\vb). 
\]
\end{thm}
\begin{proof}
  We proceed by strong induction on \( N:=n+|\lm|+\sum_{i=1}^n s_i\ge1 \), where
  it is unnecessary to check the base case.
  Suppose that the theorem holds for
  all integers less than \( N \). If \( \mu=\lambda \), then it follows from
  Lemma~\ref{lem:G=1}. Therefore we may assume \( \mu\ne\lambda \). We consider
  the following two cases.
  
  \textbf{Case 1:} Suppose that there is an integer \( 1\le t\le n-1 \)
  satisfying \( \mu_t\ge\lambda_{t+1} \). We set \( \gamma^{(1)} =
  (\gamma_1,\dots, \gamma_t) \) and \( \gamma^{(2)} =
  (\gamma_{t+1},\dots,\gamma_n) \) for each \( \gamma\in\{\lambda,\mu,\vr,\vs\}
  \). By Lemma \ref{lem:rec1}, the statement is separated into the cases \(
  (\lambda^{(i)},\mu^{(i)},\vr^{(i)},\vs^{(i)}) \) for \( i=1,2 \). Each case is
  then covered by the induction hypothesis.

  \textbf{Case 2:} Suppose that \( \mu_t<\lambda_{t+1} \) for all \( 1\le t\le
  n-1 \) (or \(n=1\)). Then by the assumption in the theorem we have \(
  \vr,\vs\in\RPar_n \). Moreover, since \( \lambda\in\Par_n \), we have \(
  \mu_i<\lambda_i \) for all \( 1\le i\le n \) (if \( n=1 \), this follows from
  the assumptions \( \mu\subseteq\lambda \) and \( \mu\ne\lambda \)). Since \(
  \vs\in\RPar_n \) (with \( s_0=0 \)), we can find a unique integer $1\le k\le
  n$ satisfying $s_{k-1}<s_k =\dots=s_n$. If \( s_k<r_k \), then the statement
  follows from Lemma~\ref{lem:rec4}. Thus we now assume \( s_k\ge r_k \).

  If $\lambda_{k}>\lambda_{k+1}$, then we apply Lemma~\ref{lem:rec2}.
  If \( \lambda_{k}=\lambda_{k+1} \), then \( k<n \) because otherwise
  we would have \( 0\le \mu_n<\lambda_n = \lambda_{n+1} =0 \), which is a
  contradiction. Therefore we can apply Lemma~\ref{lem:rec3}. In
  either case the statement reduces to some cases with a smaller
  \( N \), and therefore the induction hypothesis applies.

  The above two cases show that the statement holds for \( N \), and the theorem
  follows by induction.
\end{proof}

\subsection{Proof of Theorem~\ref{thm:main_G_col}: column-flagged Grothendieck polynomials}
\label{sec:proof-theor-refthm:m-2}

In this subsection we prove the Jacobi--Trudi-like identity for \( G_\lmc^{\col(\vr, \vs)}(\vx;\va,\vb)
\) in Theorem~\ref{thm:main_G_col}. For \( \lambda,\mu\in\Par_n \) and \(
\vr,\vs\in\NN^n \), let
\begin{equation}
  \label{eq:tGc}
  \tG_\lmc^{\col(\vr, \vs)}(\vx;\va,\vb) 
  = D \det\left(   e_{\lambda_i-\mu_j-i+j}
    [X_{[r_j,s_i]}\ominus(A_{i-1}-A_{j}-B_{\lambda_i} + B_{\mu_j})]\right)_{i,j=1}^n,
\end{equation}
where \( D=\prod_{i=1}^n\prod_{l=r_i}^{s_i}(1-\alpha_ix_l)^{-1} \). Then the
identity in Theorem~\ref{thm:main_G_col} is written as
\[ 
  G_\lmc^{\col(\vr, \vs)}(\vx;\va,\vb)=\tG_\lmc^{\col(\vr, \vs)}(\vx;\va,\vb).
\]
Since we prove this identity in an analogous way as in the previous subsection, most of the 
proofs in this subsection will be omitted and some additional explanations will be
given when necessary.

As before we need several lemmas.

\begin{lem}\label{lem:rec1c} Let \( \lambda,\mu\in\Par_n,\vr,\vs\in\NN^n \)
  with \( \mu\subseteq\lambda \) and suppose \( 1\le k\le n-1\) is an integer
  such that \( \mu_k\ge\lambda_{k+1} \). 
  If \( \phi_{\va} \) is the map defined by shifting \( \alpha_i \) to \( \alpha_{i+1} \), then
\begin{align*}
  G_\lmc^{\col(\vr,\vs)}(\vx;\va,\vb)
  &=G_{({\lambda}^{(1)})'/({\mu}^{(1)})'}^{\col({\vr}^{(1)},{\vs}^{(1)})}(\vx;\va,\vb)\cdot\phi_{\va}^k
    \left(G_{({\lambda}^{(2)})'/({\mu}^{(2)})'}^{\col({\vr}^{(2)},{\vs}^{(2)})}(\vx;\va,\vb)\right),\\
  \tG_\lmc^{\col(\vr,\vs)}(\vx;\va,\vb)
  &=\tG_{({\lambda}^{(1)})'/({\mu}^{(1)})'}^{\col({\vr}^{(1)},{\vs}^{(1)})}(\vx;\va,\vb)\cdot\phi_{\va}^k
    \left(\tG_{({\lambda}^{(2)})'/({\mu}^{(2)})'}^{\col({\vr}^{(2)},{\vs}^{(2)})}(\vx;\va,\vb)\right),
\end{align*}
where for a sequence \( \vt=(t_1,\ldots,t_n) \) we denote
$\vt^{(1)}=(t_1,\ldots,t_k)$ and $\vt^{(2)}=(t_{k+1},\ldots,t_n)$.
\end{lem}
\begin{proof}
  It is similar to the proof of Lemma~\ref{lem:rec1}.
\end{proof}

The following two lemmas show that \(
G_{\lambda'/\mu'}^{\col(\vr,\vs)}(\vx;\va,\vb) \)
and \( \tG_{\lambda'/\mu'}^{\col(\vr,\vs)}(\vx;\va,\vb) \) have the same
initial conditions.

\begin{lem}\label{lem:G=1c}
  For \( \lambda\in\Par_n,\vr,\vs\in\NN^n \),
  \[
    G_{\lambda'/\lambda'}^{\col(\vr,\vs)}(\vx;\va,\vb) = 
    \tG_{\lambda'/\lambda'}^{\col(\vr,\vs)}(\vx;\va,\vb) = 1.
  \]
\end{lem}
\begin{proof}
  It is similar to the proof of Lemma~\ref{lem:G=1}.
\end{proof}

\begin{lem}\label{lem:rec4c}
Let \( \lambda,\mu\in\Par_n,\vr,\vs\in\PP^n \) with \( \mu\subseteq\lambda \).
Suppose \( r_i-\mu_i\le r_{i+1}-\mu_{i+1}, s_i-\lambda_i\le s_{i+1}-\lambda_{i+1}+1 \).  If \( s_k-\lambda_k+1<r_k-\mu_k \) and \( \mu_k<\lambda_k \) for some 
\( 1\le k\le n \), then
\[
  G_{\lmc}^{\col(\vr,\vs)}(\vx;\va,\vb)= \tG_{\lmc}^{\col(\vr,\vs)}(\vx;\va,\vb)=0.
\]
\end{lem}
\begin{proof}  
  We first show that \( G^{\col(\vr,\vs)}_{\lmc}(\vx;\va,\vb)=0 \).
  To see this, suppose \( T\in \MMSVT^{\col(\vr,\vs)}(\lmc) \). 
  The \( k \)th column of \( T \) consists of \( \lambda_k-\mu_k \) cells,
  so at least \( \lambda_k-\mu_k \) distinct integers appear in the column.
  But this is impossible since \( s_k-\lambda_k+1<r_k-\mu_k \).
  Thus \(\MMSVT^{\col(\vr, \vs)}(\lmc)=\emptyset\) and
  \( G_{\lmc}^{\col(\vr,\vs)}(\vx;\va,\vb)=0 \).
  
  For the second equality, we claim that for all \( 1\le
  i\le k \) and \( k\le j\le n \), the \( (i,j) \)-entry of the matrix in
  \eqref{eq:tGc} is \( 0, \) that is,
    \begin{equation}\label{eq:Eijc}
      \sum_{\ell\ge0 }e_{\lambda _{i}-\mu _{j}-i+j+\ell }\left[
        X_{[r_{j},s_{i}]}\right] e_{\ell }\left[ A_{i-1}-A_{j}-B_{\lambda_i} + B_{\mu_j}\right]=0.
    \end{equation}
    Observe that \( \lambda_i-\mu_j-i+j \ge \lambda_k-\mu_k>0 \). By the
    assumption \( \mu_m<\lambda_{m+1} \) for all \( m<n \) (combined with the
    assumption of the lemma), we have \( r_k-\mu_k\le r_j-\mu_j \)
    and \( s_i-\lambda_i\le s_k-\lambda_k+(k-i) \). Thus
    \[ 
      s_i-\lambda_i\le s_k -\lambda_k+k-i<r_k-\mu_k+k-i-1\le r_j-\mu_j+k-i-1\le r_j-\mu_j+j-i-1,
    \]
    and we obtain \( s_i-r_j+1< \lambda_i-\mu_j-i+j \). This implies 
    \( e_{\lambda_i-\mu_j-i+j+\ell}[X_{[r_j,s_i]}]=0\) for all \( \ell\ge0 \), because the number of variables
    is \(\max(0,s_i-r_j+1) \), which is smaller than \( \lambda_i-\mu_j-i+j+\ell \). 
    Therefore the claim \eqref{eq:Eijc} holds and we obtain \( \tG_{\lmc}^{\col(\vr,\vs)}(\vx;\va,\vb)=0 \).
\end{proof}

Recall that for \( \lambda\in \Par_n \) and \( \vs\in \RPar_n \)
we assume \( \lambda_{n+1}=0 \) and \( s_0=0 \).
The following two lemmas show that \(
G_{\lambda'/\mu'}^{\col(\vr,\vs)}(\vx;\va,\vb) \)
and \( \tG_{\lambda'/\mu'}^{\col(\vr,\vs)}(\vx;\va,\vb) \) have the same
recursions. 

\begin{lem}\label{lem:rec2c} 
  Let \( \lambda,\mu\in\Par_n,\vr,\vs\in \PP^n \). Suppose \( r_i-\mu_i\le r_{i+1}-\mu_{i+1}\) for all \( 1\le i \le n-1 \). Let \( 1\le
  k\le n \) be an integer satisfying the following conditions:
  \begin{enumerate}
  \item \( \mu_k<\lambda_k\),
  \item \( r_k-\mu_k\le s_k-\lambda_k+1 \),
  \item \( s_{k-1}-\lambda_{k-1}\le s_{k}-\lambda_{k} \) if \( k\ge2 \), and
  \item \( \lambda_k>\lambda_{k+1} \).
  \end{enumerate}
  Then
\begin{align*}
  G_\lmc^{\col(\vr,\vs)}(\vx;\va,\vb)
  &=\frac{1-x_{s_k}\beta_{\lambda_k}}{1-x_{s_k}\alpha_k}G_{\lmc}^{\col(\vr,\vs-\epsilon_k)}(\vx;\va,\vb)
    +\frac{x_{s_k}}{1-x_{s_k}\alpha_k}G_{(\lambda-\epsilon_k)'/\mu'}^{\col(\vr,\vs-\epsilon_k)}(\vx;\va,\vb),\\
  \tG_\lmc^{\col(\vr,\vs)}(\vx;\va,\vb)
  &=\frac{1-x_{s_k}\beta_{\lambda_k}}{1-x_{s_k}\alpha_k}\tG_{\lmc}^{\col(\vr,\vs-\epsilon_k)}(\vx;\va,\vb)
    +\frac{x_{s_k}}{1-x_{s_k}\alpha_k}\tG_{(\lambda-\epsilon_k)'/\mu'}^{\col(\vr,\vs-\epsilon_k)}(\vx;\va,\vb).
\end{align*}
Moreover, if \( s_k=1 \), then
\begin{align}
  \label{eq:Gcol1}
  G_{\lambda'/\mu'}^{\col(\vr,\vs-\epsilon_k)}(\vx;\va,\vb)
  &=0,\\
  \label{eq:Gcol2}
  \tG_{\lambda'/\mu'}^{\col(\vr,\vs-\epsilon_k)}(\vx;\va,\vb)
  &=0,\\
  \label{eq:Gcol3}
  G_{(\lambda-\epsilon_k)'/\mu'}^{\col(\vr,\vs-\epsilon_k)}(\vx;\va,\vb) 
 &= 
\begin{cases}
 0 & \mbox{if \( \lambda_k>\mu_k+1 \)},\\
 G_{(\lambda-\epsilon_k)'/\mu'}^{\col(\vr,\vs)}(\vx;\va,\vb) & \mbox{if \( \lambda_k=\mu_k+1 \),}
\end{cases}\\
  \label{eq:Gcol4}
  \tG_{(\lambda-\epsilon_k)'/\mu'}^{\col(\vr,\vs-\epsilon_k)}(\vx;\va,\vb) 
 &= 
\begin{cases}
 0 & \mbox{if \( \lambda_k>\mu_k+1 \)},\\
 \tG_{(\lambda-\epsilon_k)'/\mu'}^{\col(\vr,\vs)}(\vx;\va,\vb) & \mbox{if \( \lambda_k=\mu_k+1 \).}
\end{cases}
\end{align}
\end{lem}

\begin{proof}
  One can prove the first identity with the same argument as in the
  proof of Lemma~\ref{lem:rec2}. The only thing that needs a careful
  thought is whether adding any number of integers \( s_k \) in the
  cell \( (\lambda_k,k) \) of a tableau in
  \( \MMSVT^{\col(\vr, \vs-\epsilon_k)}(\lambda'/\mu') \) or
  \( \MMSVT^{\col(\vr, \vs-\epsilon_k)}((\lambda-\epsilon_k)'/\mu') \)
  always gives a tableau in
  \( \MMSVT^{\col(\vr, \vs)}(\lambda'/\mu') \). This is indeed true
  because \( s_k\ge r_k \) (which follows from \( \mu_k<\lambda_k\)
  and \( r_k-\mu_k\le s_k-\lambda_k+1 \)) and the maximum possible
  entry of \( (\lambda_k,k-1) \) in a tableau of
  \( \MMSVT^{\col(\vr, \vs)}(\lambda'/\mu') \) is
  \( s_{k-1}-\lambda_{k-1}+\lambda_k \le s_k \) (because of the
  assumption \( s_{k-1}-\lambda_{k-1}\le s_k-\lambda_k \)). Note that
  if \( k=1 \), we do not have to consider the cell
  \( (\lambda_k,k-1) \) because
  \( (\lambda_k,k-1)\not\in \lambda'/\mu' \). Moreover, if
  \( s_k=1 \), we obtain \eqref{eq:Gcol1} and \eqref{eq:Gcol3} by
  Remark~\ref{rem:s_k=1}. Note that if \( s_k=1 \) and
  \( \lambda_k=\mu_k+1 \), then
  \(
  G_{(\lambda-\epsilon_k)'/\mu'}^{\col(\vr,\vs-\epsilon_k)}(\vx;\va,\vb)
  = G_{(\lambda-\epsilon_k)'/\mu'}^{\col(\vr,\vs)}(\vx;\va,\vb) \)
  because there are no cells in column \( k \) of
  \( (\lambda-\epsilon_k)'/\mu' \), hence the bound condition for
  column \( k \) is irrelevant.

  For the second identity let
  \begin{align*}
    E_{i,j} &= e_{\lambda_i-\mu_j-i+j} [X_{[r_j,s_i]}\ominus(A_{i-1}-A_{j}-B_{\lambda_i} + B_{\mu_j})],\\
    F_{i,j} &= e_{\lambda_i-\mu_j-i+j} [X_{[r_j,s_i-\chi(i=k)]}\ominus(A_{i-1}-A_{j}-B_{\lambda_i} + B_{\mu_j})],\\
    G_{i,j} &= e_{\lambda_i-\mu_j-i+j-\chi(i=k)} [X_{[r_j,s_i-\chi(i=k)]}\ominus(A_{i-1}-A_{j}-B_{\lambda_i-\chi(i=k)} + B_{\mu_j})].
  \end{align*}
  Then it suffices to show that
\begin{equation}\label{eq:EFGc}
  E_{k,j}=(1-x_{s_k}\beta_{\lambda_k})F_{k,j}+x_{s_k} G_{k,j}
\end{equation}
for all \( 1\le j\le n\). In order to prove \eqref{eq:EFGc}, 
we consider the two cases  \( r_j\le s_k\) and \( r_j> s_k\). Let \( m=\lambda _{k}-\mu _{j}-k+j \) and \( Z=A_{k-1}-A_{j}-B_{\lambda_k} + B_{\mu_j}\).

First, suppose \( r_j\le s_k\). We use the same argument in the proof
of Lemma~\ref{lem:rec2}. Since
\( e_{m+\ell}[X_{[r_{j},s_{k}-1]}] =
e_{m+\ell}[X_{[r_{j},s_{k}]}-x_{s_k}] \), by Lemma~\ref{lem:e_m[Z]},
we have
\begin{align*}
  e_{m+\ell }\left[X_{[r_{j},s_{k}]}\right]
  &= e_{m+\ell }\left[X_{[r_{j},s_{k}-1]}\right] + x_{s_k}e_{m+\ell-1}\left[X_{[r_{j},s_{k}-1]}\right],\\
  e_{\ell}[Z]
  &= e_{\ell }[Z+\beta_{\lambda_k}] -\beta_{\lambda_k}e_{\ell-1}[Z].
\end{align*}
Thus
\begin{align*}
  E_{k,j}&=\sum_{\ell\ge0 }e_{m+\ell }\left[
    X_{[r_{j},s_{k}-1]}\right] e_{\ell }[Z]
  +x_{s_k} \sum_{\ell\ge0 }e_{m+\ell-1}\left[
    X_{[r_{j},s_{k}-1]}\right] e_{\ell }[Z]\\
         &= F_{k,j}+x_{s_k} \sum_{\ell\ge0 }e_{m+\ell-1}\left[
           X_{[r_{j},s_{k}-1]}\right] \left(e_{\ell }[Z+\beta_{\lambda_k}] -\beta_{\lambda_k}e_{\ell-1}[Z]\right)\\
         &= F_{k,j}+x_{s_k} G_{k,j} -x_{s_k}\beta_{\lambda_k} F_{k,j},
\end{align*}
as required.

 Suppose \( r_j>s_k\). Since \(
X_{[r_j,s_k]}=X_{[r_j,s_k-1]}=0 \), we obtain 
\[
  E_{k,j} = F_{k,j} = e_{m}[0\ominus Z] = e_{-m}[Z],  \qquad G_{k,j} = e_{1-m}[Z+\beta_{\lambda_k}]. 
\]
 By the assumption \( \mu_k<\lambda_k \) in the lemma,
we have \( j\ne k \) because otherwise 
\( s_k<r_j=r_k\le s_k +\mu_k-\lambda_k+1\le s_k\), which is impossible. If \( j> k \), 
then \( \mu_j\le\mu_k<\lambda_k \). Thus
\[ 
  1-m=(1+\mu_j-\lambda_k)+(k-j)\le 0+(-1)<0,
\]
which implies \( E_{k,j} = F_{k,j} = G_{k,j}=0 \). Now suppose \( j<k \).  Since 
 \( r_j-\mu_j\le r_k-\mu_k\le s_k-\lambda_k+1 \), we obtain \(s_k<r_j\le s_k+\mu_j-\lambda_k+1. \) 
 This gives \( \mu_j-\lambda_k\ge0 \) and hence
\[
  e_{-m}[Z] =  e_{-m} [A_{k-1}-A_{j}-B_{\lambda_k} + B_{\mu_j}]
  =e_{-m}(\alpha_{j+1},\dots,\alpha_{k-1},\beta_{\lambda_{k}+1},\dots,\beta_{\mu_j})=0,
\]
because the number of variables is \( k-1-j+\mu_j-\lambda_k=-m-1<-m \)
(which also implies \( -m>0 \)).
Therefore \( E_{k,j} = F_{k,j} = e_{-m}[Z]=0 \) and similarly we obtain \(
G_{k,j} = e_{1-m}[Z+\beta_{\lambda_k}]=0 \). This shows \eqref{eq:EFGc}.

To complete the proof, it remains to show
\eqref{eq:Gcol2} and \eqref{eq:Gcol4} for the case \( s_k=1 \). Since
\( r_j>s_k-1=0 \), \eqref{eq:Gcol2} and \eqref{eq:Gcol4} for the case
\( \lambda_k>\mu_k+1 \) can be shown by applying the same argument as
before. Suppose \( s_k=1 \) and \( \lambda_k=\mu_k+1 \).
Then, since \(  \mu_k = (\lambda-\epsilon_k)_k \), we have
\( \mu_{k-1}\ge \mu_k= (\lambda-\epsilon_k)_k \)
and
\( \mu_{k} = (\lambda-\epsilon_k)_k \ge (\lambda-\epsilon_k)_{k+1} \).
Therefore, by applying Lemma~\ref{lem:rec1c} twice, we obtain
\begin{multline*}
    \tG_{(\lambda-\epsilon_k)'/\mu'}^{\col(\vr,\vs-\epsilon_k)}(\vx;\va,\vb)
    =\tG_{({\lambda}^{(1)})'/({\mu}^{(1)})'}^{\col({\vr}^{(1)},{\vs}^{(1)})}(\vx;\va,\vb)\\
    \times
    \phi_{\va}^{k-1}
    \left( \tG_{((\lambda_k-1))'/((\mu_k))'}^{\col((r_k),(s_k-1))}(\vx;\va,\vb) \right)
    \phi_{\va}^{k}
    \left( \tG_{({\lambda}^{(2)})'/({\mu}^{(2)})'}^{\col({\vr}^{(2)},{\vs}^{(2)})}(\vx;\va,\vb) \right),
\end{multline*}
where \( \vr^{(1)} = (r_1,\dots,r_{k-1}) \) and
\( \vr^{(2)} = (r_{k+1},\dots,r_{n}) \), and \( \vs^{(1)} \),
\( \vs^{(2)} \), \( \lambda^{(1)} \), \( \lambda^{(2)} \),
\( \mu^{(1)} \), and \( \mu^{(2)} \) are defined similarly. Since
\( (\lambda_k-1)=(\mu_k) \) as partitions, by Lemma~\ref{lem:G=1c},
the above identity can be rewritten as
\begin{equation}\label{eq:3}
    \tG_{(\lambda-\epsilon_k)'/\mu'}^{\col(\vr,\vs-\epsilon_k)}(\vx;\va,\vb)
    =\tG_{({\lambda}^{(1)})'/({\mu}^{(1)})'}^{\col({\vr}^{(1)},{\vs}^{(1)})}(\vx;\va,\vb)
    \phi_{\va}^{k}
    \left( \tG_{({\lambda}^{(2)})'/({\mu}^{(2)})'}^{\col({\vr}^{(2)},{\vs}^{(2)})}(\vx;\va,\vb) \right).
\end{equation}
The same argument also shows that \(  \tG_{(\lambda-\epsilon_k)'/\mu'}^{\col(\vr,\vs-\epsilon_k)}(\vx;\va,\vb) \)
is equal to the right-hand side of \eqref{eq:3}. Therefore we obtain \eqref{eq:Gcol4}.
\end{proof}

\begin{lem}\label{lem:rec3c}
  Let \( \lambda,\mu\in\Par_n\) and \(\vr,\vs\in \PP^n \). Suppose \( r_i-\mu_i\le r_{i+1}-\mu_{i+1}\) for all \( 1\le i \le n-1 \). Let \( 1\le k\le n-1 \) be an integer
satisfying the following conditions:
\begin{enumerate}
\item \( \mu_k<\lambda_k\),
\item \( r_k-\mu_k\le s_k-\lambda_k+1 \),
\item \( s_{k}-1=s_{k+1} \), and
\item \( \lambda_k=\lambda_{k+1} \).
\end{enumerate}
Then
\begin{align*}
G_{\lmc}^{\col(\vr,\vs)}(\vx;\va,\vb)&=G_{\lmc}^{\col(\vr,\vs-\epsilon_k)}(\vx;\va,\vb),\\
\tG_{\lmc}^{\col(\vr,\vs)}(\vx;\va,\vb)&=\tG_{\lmc}^{\col(\vr,\vs-\epsilon_k)}(\vx;\va,\vb).
\end{align*}
Moreover, if \( s_k=1 \), then
\[
  G_{\lambda'/\mu'}^{\col(\vr,\vs-\epsilon_k)}(\vx;\va,\vb) =
  \tG_{\lambda'/\mu'}^{\col(\vr,\vs-\epsilon_k)}(\vx;\va,\vb) =0.
\]
\end{lem}

\begin{proof} 
  The proof is similar to that of Lemma~\ref{lem:rec2c}.
  Let \( T\in \MMSVT^{\col(\vr,\vs)}(\lambda'/\mu') \). Then 
  \[
    \max(T(\lambda_k,k))\le \min(T(\lambda_k,k+1)) = \min(T(\lambda_{k+1}, k+1)) \le
    s_{k+1} = s_k - 1.
  \]
  Thus every integer in the \( k \)th column of \( T \) is at most
  \( s_k-1 \), which shows the first identity. As before, we have
  \( G_{\lambda'/\mu'}^{\col(\vr,\vs-\epsilon_k)}(\vx;\va,\vb)=0 \) if
  \( s_k=1 \) by Remark~\ref{rem:s_k=1}.
  
  For the second identity let
    \begin{align*}
    E_{i,j} &= e_{\lambda_i-\mu_j-i+j} [X_{[r_j,s_i]}\ominus(A_{i-1}-A_{j}-B_{\lambda_i} + B_{\mu_j})],\\
    F_{i,j} &= e_{\lambda_i-\mu_j-i+j} [X_{[r_j,s_i-\chi(i=k)]}\ominus(A_{i-1}-A_{j}-B_{\lambda_i} + B_{\mu_j})].
  \end{align*}
  Then it suffices to show that 
  \begin{equation}\label{eq:EFc}
    E_{k,j}=(1-x_{s_k}\alpha_k)F_{k,j} + x_{s_{k}}F_{k+1,j}
  \end{equation}
  for all \( 1\le j \le n \). In order to prove \eqref{eq:EFc}, 
  we consider the two cases \( r_j\le s_k\) and \( r_j> s_k\). Let \( m=\lambda _{k}-\mu _{j}-k+j \) and \( Z=A_{k-1}-A_{j}-B_{\lambda_k} + B_{\mu_j}\).

  First, suppose \( r_j\le s_k\). By Lemma~\ref{lem:e_m[Z]},
  we have
  \begin{align*}
    e_{m+\ell}\left[ X_{[ r_j,s_k ]} \right] &=  e_{m+\ell}\left[ X_{[ r_j,s_k-1 ]} \right] + x_{s_k} e_{m+\ell-1}\left[ X_{[ r_j,s_k-1 ]} \right], \\
    e_{\ell}[Z] &=e_\ell[Z+\alpha_k] - \alpha_k e_{\ell-1}[Z].
  \end{align*}
  Thus
 \begin{align*}
  E_{k,j}&=\sum_{\ell\ge0 }e_{m+\ell }\left[
    X_{[r_{j},s_{k}-1]}\right] e_{\ell }[Z]
  +x_{s_k} \sum_{\ell\ge0 }e_{m+\ell-1}\left[
    X_{[r_{j},s_{k}-1]}\right] e_{\ell }[Z]\\
         &= F_{k,j}+x_{s_k} \sum_{\ell\ge0 }e_{m+\ell-1}\left[
           X_{[r_{j},s_{k}-1]}\right] \left(e_{\ell }[Z+\alpha_k] -\alpha_ke_{\ell-1}[Z]\right)\\
         &= F_{k,j}+x_{s_k} F_{k+1,j} -x_{s_k}\alpha_k F_{k,j},
\end{align*}
as required.

 Suppose \( r_j>s_k\). Since \(
X_{[r_j,s_k]}=X_{[r_j,s_k-1]}=0 \), we obtain 
\[
  E_{k,j} = F_{k,j} = e_{m}[0\ominus Z] = e_{-m}[Z],
  \qquad F_{k+1,j} = e_{m-1}[0\ominus Z] = e_{1-m}[Z].
\]
Then by the same argument as in the proof of Lemma~\ref{lem:rec2c}, we
obtain \( E_{k,j} = F_{k,j} = F_{k+1,j} =0 \), which shows
\eqref{eq:EFc}.

Finally, if \( s_k = 1 \), then \( r_j > s_k-1 = 0 \) and hence we
obtain \( F_{k,j} = 0 \) by the argument above. This shows
\( \tG_{\lambda'/\mu'}^{\col(\vr,\vs-\epsilon_k)}(\vx;\va,\vb) =0 \)
for \( s_k = 1 \), which completes the proof.
\end{proof}

We are ready to prove Theorem~\ref{thm:main_G_col}, which is restated as follows.

\begin{thm}
  Let \( \lambda,\mu\in\Par_n,\vr,\vs\in\PP^n \) with \( \mu\subseteq\lambda \).
If $r_i-\mu_i\le r_{i+1}-\mu_{i+1}$ and $s_i-\lambda_i\le s_{i+1}-\lambda_{i+1}+1$ whenever
$\mu_i<\lambda_{i+1}$ for \( 1\le i\le n-1 \), then
\[ 
  G_{\lmc}^{\col(\vr,\vs)}(\vx;\va,\vb)=\tG_{\lmc}^{\col(\vr,\vs)}(\vx;\va,\vb). 
\]
\end{thm}

\begin{proof}
  The proof is similar to that of Theorem~\ref{thm:row_f_G}. We proceed by
  strong induction on \( N:=n+|\lm|+\sum_{i=1}^n s_i \). Suppose that the
  theorem holds for all integers less than \( N \). If \( \mu=\lambda \), then
  it follows from Lemma~\ref{lem:G=1c}. Therefore we may assume \( \mu\ne\lambda
  \). We consider the following two cases.
  
  \textbf{Case 1:} Suppose that there is an integer \( 1\le t\le n-1 \)
  satisfying \( \mu_t\ge\lambda_{t+1} \). We set \( \gamma^{(1)} =
  (\gamma_1,\dots, \gamma_t) \) and \( \gamma^{(2)} =
  (\gamma_{t+1},\dots,\gamma_n) \) for each \( \gamma\in\{\lambda,\mu,\vr,\vs\}
  \). By Lemma~\ref{lem:rec1c}, the statement is separated into the cases \(
  (\lambda^{(i)},\mu^{(i)},\vr^{(i)},\vs^{(i)}) \) for \( i=1,2 \). Each case is
  then covered by the induction hypothesis.

  \textbf{Case 2:} Suppose that \( \mu_t<\lambda_{t+1} \) for all
  \( 1\le t\le n-1 \) (or \(n=1\)). Then by the assumption, we have
  $r_i-\mu_i\le r_{i+1}-\mu_{i+1}$ and
  $s_i-\lambda_i\le s_{i+1}-\lambda_{i+1}+1$ for all \( 1\le i\le n-1 \).
  Note also that \( \mu_i<\lambda_i \) for all \( 1\le i\le n \)
  since \( \mu_i<\lambda_{i+1}\le \lambda_i \) for \( 1\le i\le n-1 \)
  and \( \mu_n\le \mu_{n-1}<\lambda_n \).

  Now let
  \[ 
    k=\max\left( \{1\}\cup \{2\le m\le n:s_{m-1}-\lambda_{m-1}\le s_m-\lambda_m\} \right).
  \]
  Observe that, by the definition of \( k \),
  we have \( s_{k-1}-\lambda_{k-1}\le s_k-\lambda_k \) if \( k\ge2 \).
  Moreover, since 
  \(s_i-\lambda_i\le s_{i+1}-\lambda_{i+1}+1\) for all \( 1\le i\le n-1 \),
  we have 
  \begin{equation}\label{eq:1}
    s_k-\lambda_k=s_{k+1}-\lambda_{k+1}+1 \qquad \mbox{if \( k < n \).}
  \end{equation}

  If \( s_k-\lambda_k+1<r_k-\mu_k \), then the statement holds by
  Lemma~\ref{lem:rec4c}. Hence we may assume
  \( r_k-\mu_k\le s_k-\lambda_k+1 \). If $\lambda_{k}>\lambda_{k+1}$,
  then we apply Lemma~\ref{lem:rec2c}. If
  \( \lambda_{k}=\lambda_{k+1} \), then \( k<n \) because otherwise
  \( 0\le \mu_n<\lambda_n = \lambda_{n+1} =0 \), which is a
  contradiction. In this case, by \eqref{eq:1}, we have
  \( s_k=s_{k+1}-\lambda_{k+1}+\lambda_k+1=s_{k+1}+1 \). Thus we can
  apply Lemma \ref{lem:rec3c}. In either case the statement reduces to
  some cases with a smaller \( N \). Thus, by the induction
  hypothesis, we obtain the statement.

  The above two cases show that the statement holds for \( N \), and the theorem
  follows by induction.
\end{proof}

\begin{remark}
  Computer experiments suggest that the condition \( r_i-\mu_i\le
  r_{i+1}-\mu_{i+1} \) in Theorem~\ref{thm:main_G_col} may be replaced by the
  weaker condition \( r_i-\mu_i\le r_{i+1}-\mu_{i+1}+1 \) as in a result of
  Wachs \cite[Theorem~3.5*]{Wachs_1985}. Under this weaker condition, the inequality  
  \( \mu_j-\lambda_k\ge0 \) does not hold when \( j<k \) in the last part of 
  the proof of Lemma~\ref{lem:rec2c}.
  Thus we must also consider the case \( \mu_j<\lambda_k \) when \( j<k \) 
  to complete the proof of Lemma~\ref{lem:rec2c} under the weaker condition. 
  All other lemmas in this subsection containing the condition 
  \( r_i-\mu_i\le r_{i+1}-\mu_{i+1} \) can be proved similarly under the weaker condition 
  \( r_i-\mu_i\le r_{i+1}-\mu_{i+1}+1 \).
\end{remark}

\subsection{Matsumura's flagged Grothendieck polynomials}
\label{sec:mats-flagg-groth}

Matsumura \cite{Matsumura_2018} found a Jacobi--Trudi-like formula for flagged
Grothendieck polynomials. In this subsection we show that our Jacobi--Trudi-like formula
implies his result.

Following the notation in \cite{Matsumura_2018}, 
for \( \lambda,\mu\in\Par_n \) and \( f=(f_1,\dots,f_n), g=(g_1,\dots,g_n)\in \PP^n \),
let
\[
  G_{\lambda/\mu,f/g}(\vx): = \sum_{T\in \FSVT(\lambda/\mu,f/g)} \beta^{|T|-|\lm|} \vx^T,
\]
where $\FSVT(\lambda/\mu,f/g)$ is the set of set-valued tableaux $T$ of shape
$\lm$ such that every integer appearing in row $i$ is in
$\{g_i,g_i+1,\dots,f_i\}$,
and \( |T| \) is the total number of integers in \( T \).

We define $G^{[p/q]}_m(\vx)$ by the following Laurent series in \( u \):
\begin{equation}\label{eq:Gp/q}
  \sum_{m\in\ZZ} G^{[p/q]}_m(\vx) u^m = \frac{1}{1+\beta u^{-1}} \prod_{q\le i\le p}
  \frac{1+\beta x_i}{1-x_iu},
\end{equation}
where
\[
  \frac{1}{1+\beta u^{-1}} =\frac{1}{1-(-\beta u^{-1})} = 1+(-\beta u^{-1})+(-\beta u^{-1})^2+\cdots,
\]
and
\[
  \frac{1+\beta x_i}{1-x_iu} = (1+\beta x_i)(1 + x_i u + x_i^2 u^2 + \cdots).
\]
If \( p<q \), then the product in \eqref{eq:Gp/q} is defined to be \( 1 \). We
note that the definition of $G^{[p/q]}_m(\vx)$ in \cite[p.224]{Matsumura_2018} has a
typo, $\beta^{-1} u$, which should be $\beta u^{-1}$.

We now state Matsumura's theorem.

\begin{thm}\cite[Theorem~4]{Matsumura_2018}
\label{thm:Matsumura}
Let \( \lambda,\mu\in\Par_n \) and \( f=(f_1,\dots,f_n), g=(g_1,\dots,g_n)\in \PP^n \) with \(\mu\subseteq\lambda\) and \(g_i\le g_{i+1}, f_i\le f_{i+1}\) whenever \(\mu_i<\lambda_{i+1}\).
We have
  \begin{equation}
    \label{eq:Matsumura}
  G_{\lambda/\mu,f/g}(\vx) =
  \det \left( \sum_{s\ge0} \binom{i-j}{s} \beta^s G^{[f_i/g_j]}_{\lambda_i-\mu_j-i+j+s}(\vx) \right)_{i,j=1}^{\ell(\lambda)}.
  \end{equation}
\end{thm}

In what follows we reformulate Theorem~\ref{thm:Matsumura} using our notation \( \ominus \).
Recall that \( h_n(\vx)=0 \) for \( n<0 \).

\begin{lem}\label{lem:Gp/q}
  For any integer \( m\) and any positive integers \( p \) and \( q \), we have
  \[
    G^{[p/q]}_m(\vx) = \prod_{i=q}^p (1+\beta x_i) \sum_{k\ge0}(-\beta)^k h_{m+k}[X_{[q,p]}].
  \] 
\end{lem}
\begin{proof}
  Since \( h_n(\vx)=0 \) for \( n<0 \), we have
  \[
    \prod_{q\le i\le p}\frac{1}{1-x_iu}=\sum_{l \geq 0} h_l[X_{[q,p]}]u^l 
    =\sum_{l\in \ZZ} h_l[X_{[q,p]}]u^l.
  \]
Thus the right hand side of \eqref{eq:Gp/q} is 
\begin{align*}
  \frac{1}{1+\beta u^{-1}} \prod_{q\le i\le p}\frac{1+\beta x_i}{1-x_iu}
  &=\prod_{i=q}^p (1+\beta x_i) \sum_{k\geq 0}(-\beta)^k u^{-k} \sum_{l\in\ZZ} h_l[X_{[q,p]}]u^l\\
  &=\prod_{i=q}^p (1+\beta x_i) \sum_{m\in\ZZ} \left( \sum_{k\geq 0}(-\beta)^k h_{m+k}[X_{[q,p]}] \right)u^m,
\end{align*}
which shows the lemma.
\end{proof}

Using Lemma~\ref{lem:Gp/q}, the $(i,j)$-entry of the matrix in
\eqref{eq:Matsumura} can be rewritten as follows.
Here, to avoid confusion, we write \( \overline{\beta} = -\beta \)
as a plethystic variable.

\begin{lem}\label{lem:mat-ominus}
  We have
 \[
   \sum_{s\ge0} \binom{i-j}{s} \beta^s G^{[f_i/g_j]}_{\lambda_i-\mu_j-i+j+s}(\vx)
   =C_{i,j} h_{\lambda_i-\mu_j-i+j}[X_{[g_j,f_i]}\ominus (j-i+1)\overline{\beta}],
 \]
 where $C_{i,j} = \prod_{l=g_j}^{f_i} (1+\beta x_l)$.
\end{lem}
\begin{proof}
  By Lemma~\ref{lem:Gp/q}, we have
  \begin{align}
    \notag
    \sum_{s\ge0} \binom{i-j}{s} \beta^s G^{[f_i/g_j]}_{\lambda_i-\mu_j-i+j+s}(\vx)
    &= C_{i,j} \sum_{s\ge0} \binom{i-j}{s} \beta^s
      \sum_{k\ge0}(-\beta)^k h_{\lambda_i-\mu_j-i+j+s+k}[X_{[g_j,f_i]}]\\
    \notag
    &= C_{i,j} \sum_{m\ge0} \beta^m h_{\lambda_i-\mu_j-i+j+m}[X_{[g_j,f_i]}]
      \sum_{s=0}^m \binom{i-j}{s} (-1)^{m-s} \\
    \label{eq:hbinom}
    &= C_{i,j} \sum_{m\ge0} \beta^m h_{\lambda_i-\mu_j-i+j+m}[X_{[g_j,f_i]}]
     \binom{i-j-1}{m},
  \end{align}
  where the identity $\sum_{s=0}^m \binom{n}{s}(-1)^{m-s} = \binom{n-1}{m}$ is used.
  We claim that
  \begin{equation}\label{eq:h_m[-b]}
    \beta^m  \binom{i-j-1}{m} = h_m[(j-i+1)\overline{\beta}].
  \end{equation}
  Note that, by \eqref{eq:h_m[-b]}, the right side of \eqref{eq:hbinom} is equal to
  \[
    C_{i,j} \sum_{m\ge0} h_{\lambda_i-\mu_j-i+j+m}[X_{[g_j,f_i]}]
    h_m[(j-i+1)\overline{\beta}]
    = C_{i,j} h_{\lambda_i-\mu_j-i+j}[X_{[g_j,f_i]}\ominus (j-i+1)\overline{\beta}],
  \]
  which implies the lemma. Thus it remains to prove \eqref{eq:h_m[-b]}.

  If \( i>j \), then
  \[
    \beta^m  \binom{i-j-1}{m}
    = (-1)^{m} (-\beta)^m \binom{i-j-1}{m} =
      (-1)^m e_m[(i-j-1)\overline{\beta}]
    = h_m[(j-i+1)\overline{\beta}].
  \]
  Now suppose \( i\le j \). Then, by the identity
  \( (-1)^{m}\binom{-n}{m} = \multinom{n}{m} := \binom{n+m-1}{m}\), we
  have
  \[
    \beta^m  \binom{i-j-1}{m} =
    (-\beta)^m  (-1)^{m} \binom{i-j-1}{m} =
    (-\beta)^m  \multinom{j-i+1}{m} = h_m[(j-i+1)\overline{\beta}].
  \]
  Hence we always have \eqref{eq:h_m[-b]} and the proof is completed.
\end{proof}

Using Lemma~\ref{lem:mat-ominus} we restate Theorem~\ref{thm:Matsumura} as
follows.

\begin{thm} [Matsumura's theorem reformulated]
  Let \( \lambda,\mu\in\Par_n \) and \( f=(f_1,\dots,f_n), g=(g_1,\dots,g_n)\in \PP^n \) with \(\mu\subseteq\lambda\) and \(g_i\le g_{i+1}, f_i\le f_{i+1}\) whenever \(\mu_i<\lambda_{i+1}\).
  We have
  \begin{equation}
    \label{eq:Matsumura2}
  G_{\lambda/\mu,f/g}(\vx) = C_{\ell(\lambda)} 
  \det \left(h_{\lambda_i-\mu_j-i+j}[X_{[g_j,f_i]}\ominus (j-i+1)\overline{\beta}] \right)_{i,j=1}^{\ell(\lambda)},
  \end{equation}
  where \( \overline{\beta} = -\beta \) and \(C_{\ell(\lambda)}=\prod_{i=1}^{\ell(\lambda)}\prod_{l=g_i}^{f_i}(1+\beta x_l) \).
\end{thm}

The above theorem is a special case of our result Theorem~\ref{thm:main_G_row}
with \( n=\ell(\lambda) \), \( \vr=g \), \( \vs=f \), \(\va=(0,0,\dots)\) and
\(\vb=(-\beta,-\beta,\dots)\).

\section{Flagged dual Grothendieck polynomials}
\label{sec:flagged g}

The structure of this section is similar to that of the previous section.
In this section we give a combinatorial model for the refined dual canonical
stable Grothendieck polynomials \( g_\lambda(\vx;\va,\vb) \) using marked
reverse plane partitions. To this end we introduce a flagged version of \(
g_\lambda(\vx;\va,\vb) \) using marked reverse plane partitions and prove a
Jacobi--Trudi-like formula for this flagged version, which reduces to the
Jacobi--Trudi-like formula for \( g_\lambda(\vx;\va,\vb) \) in
Theorem~\ref{thm:JT_ab_intro}. More generally, we consider two flagged versions
of \( g_\lambda(\vx;\va,\vb) \) and extend the partition \( \lambda \) to a skew
shape.

\begin{defn} \label{def:left MRPP} A \emph{reverse plane partition} of
  shape \(\lambda/\mu\) is a filling \(T\) of the cells in
  \(\lambda/\mu\) with positive integers such that the integers are
  weakly increasing in each row and in each column. A
  \emph{left-marked reverse plane partition} (or simply \emph{marked
    reverse plane partition}) of shape $\lm$ is a reverse plane
  partition $T$ of shape $\lm$ in which every entry $T(i,j)$ with
  $T(i,j)=T(i,j+1)$ may be marked. We denote by $\MRPP(\lm)$ the set
  of marked reverse plane partitions of shape $\lm$. For
  $T\in\MRPP(\lm)$, define
  \[
    \wt(T) = \prod_{(i,j)\in\lm} \wt(T(i,j)),
  \]
  where
  \begin{align*}
    \wt(T(i,j)) =
    \begin{cases}
      -\alpha_{j} & \mbox{if $T(i,j)$ is marked,}\\
      \beta_{i-1} & \mbox{if $T(i,j)$ is not marked and $T(i,j)= T(i-1,j)$,}\\
      x_{T(i,j)} & \mbox{otherwise.}
    \end{cases}
  \end{align*}
  Here, the equality $T(i,j)= T(i-1,j)$ means that their underlying integers are equal.
  See Figure~\ref{fig:MRPP'}.
\end{defn}

\begin{figure}
  \ytableausetup{mathmode, boxsize=1.7em}
\ytableaushort{\none\none12{4^*}4,\none{1^*}135,\none11,{3^*}{3^*}3}
\caption{An example of \( T\in\MRPP(\lm)\), where \( \lambda=(6,5,3,3) \) and 
\( \mu=(2,1,1) \). We have the weight \( \wt(T)=x_1x_2x_3^2x_4x_5
(-\alpha_1)(-\alpha_2)^2(-\alpha_5)\beta_1\beta_2^2 \). }
  \label{fig:MRPP'}
\end{figure}

The main goal of this section is to prove the following combinatorial
interpretation for \( g_\lambda(\vx;\va,\vb) \).

\begin{thm} 
\label{thm:MRPP}
We have
\begin{equation}\label{eq:MRPP_l}
  g_\lambda(\vx;\va,\vb) = \sum_{T\in\MRPP(\lambda)} \wt(T).
\end{equation} 
\end{thm}

\begin{remark}
  Yeliussizov \cite[Theorem 7.2]{Yeliussizov2017} found a
  combinatorial interpretation for
  \( g_\lambda^{(\alpha,\beta)}(\vx)\), which is equal to
  \(g_\lambda(\vx;\va,\vb) \) with \( \va=(-\alpha,-\alpha,\dots) \)
  and \( \vb=(\beta,\beta,\dots) \), using rim border tableaux. In
  Appendix~\ref{sec:RBT}, we show that there is a weight-preserving
  bijection between marked rim border tableaux and marked RPPs.
  Therefore, Theorem~\ref{thm:MRPP} generalizes Yeliussizov's result.
\end{remark}

Similar to our approach in the previous section, in order to prove
Theorem~\ref{thm:MRPP} we introduce a flagged version. For completeness we
again consider row- and column-flagged versions on skew shapes.

\begin{defn}\label{def:1}
  Let \( \vr=(r_1,\dots,r_n)\in\PP^n \) and \( \vs=(s_1,\dots,s_n)\in\PP^n \).
  We denote by $\MRPP^{\row(\vr, \vs)}(\lm)$ the set of $T\in \MRPP(\lm)$ such
  that $r_i \le T(i,j)\le s_i$ for all $(i,j)\in \lm$. Similarly,
  $\MRPP^{\col(\vr, \vs)}(\lm)$ denotes the set of $T\in \MRPP(\lm)$ such that
  $r_j \le T(i,j)\le s_j$ for all $(i,j)\in \lm$.
  We define
  \emph{row-flagged} and \emph{column-flagged
   refined dual canonical stable Grothendieck polynomials} by
  \begin{align*}
    g_\lm^{\row(\vr, \vs)}(\vx;\va,\vb)
    &=\sum_{T\in\MRPP^{\row(\vr, \vs)}(\lm)} \wt(T),\\
    g_\lm^{\col(\vr, \vs)}(\vx;\va,\vb)
    &=\sum_{T\in\MRPP^{\col(\vr, \vs)}(\lm)} \wt(T).
  \end{align*}
\end{defn}

Now we state Jacobi--Trudi-like formulas
for the two flagged versions of \( g_\lambda(\vx;\va,\vb) \).

\begin{thm}
\label{thm:main_g}
Let \( \lambda,\mu\in\Par_n, \vr,\vs\in \PP^n \) such that \( r_i\le r_{i+1} \) and \( s_i\le s_{i+1} \)
whenever \( \mu_{i}<\lambda_{i+1} \). Then we have
\begin{align}
  \label{eq:g_row}
  g_\lm^{\row(\vr, \vs)}(\vx;\va,\vb)
  &= \det\left(   h_{\lambda_i-\mu_j-i+j}
  [X_{[r_j,s_i]}-A_{\lambda_i-1}+A_{\mu_j}+B_{i-1} - B_{j-1}]\right)_{i,j=1}^n,\\
  \label{eq:g_col}
  g_{\lambda'/\mu'}^{\col(\vr, \vs)}(\vx;\va,\vb)
  &= \det\left(   e_{\lambda_i-\mu_j-i+j}
  [X_{[r_j,s_i]}-A_{i-1}+A_{j-1}+B_{\lambda_i-1} - B_{\mu_j}]\right)_{i,j=1}^n.
\end{align}
\end{thm}

The proof of this theorem will be given in the next two subsections.
Our proof is based on the ideas in \cite{Kim_JT22}.
We introduce new notions in Definitions~\ref{defn:MRPPI} and \ref{defn:BMRPPI},
which simplify the proofs in \cite{Kim_JT22}.

By specializing Theorem~\ref{thm:main_g} we obtain the following
corollary, which implies 
the combinatorial interpretation for \( g_\lambda(\vx_n;\va,\vb) \)
in Theorem~\ref{thm:MRPP} as \( n\to\infty \).

\begin{cor} \label{cor:g=MRPP}
  For a partition \( \lambda\in\Par_n \),
  we have
  \[
    g_\lambda(\vx_n;\va,\vb) = \sum_{T\in\MRPP(\lambda), \max(T)\le n} \wt(T),
  \]
  where \( \max(T) \) is the largest entry in \( T \).
\end{cor}
\begin{proof}
  This can be proved by the same argument as in the proof of
  Corollary~\ref{cor:G=MMSVT} using \eqref{eq:g_row} and Theorem~\ref{thm:JT_ab_intro}.
\end{proof}

\begin{remark}
  The identity~\eqref{eq:g_col} with \( \va=(0,0,\dots) \) is proved in
  \cite[Theorem~1.4]{Kim_JT22} under the weaker assumption \( r_i\le
  r_{i+1}+1 \) and \( s_i\le s_{i+1}+1 \) whenever \( \mu_{i}<\lambda_{i+1} \).
  We cannot replace the assumption in Theorem~\ref{thm:main_g} by this weaker
  one. For example, if \( \lambda=(1,1), \mu=(0,0),
  \vr=(1,1) \) and \( \vs=(2,1) \), then
  the left and right sides of \eqref{eq:g_col} are \( x_1^2 - \alpha_1 x_1 \)
  and \( x_1^2 - \alpha_1x_1 - \alpha_1x_2 \), respectively.
\end{remark}
 
The following lemmas show that we may assume \( \mu\subseteq\lambda \) and \( \vr\le\vs \) when we
prove Theorem~\ref{thm:main_g}.

\begin{lem} \label{lem:mu_subset_la}
  Under the assumption in
  Theorem~\ref{thm:main_g}, if \( \mu\not\subseteq\lambda \), then the both
  sides of \eqref{eq:g_row} and \eqref{eq:g_col} are equal to \( 0 \).
\end{lem}
\begin{proof}
  By definition, if \( \mu\not\subseteq\lambda \), we have \(
  g_{\lambda/\mu}^{\row(\vr,\vs)}(\vx;\va,\vb) =
  g_{\lambda'/\mu'}^{\col(\vr,\vs)}(\vx;\va,\vb) = 0 \). By
  Lemma~\ref{lem:det(h)=0}, the determinants in \eqref{eq:g_row} and
  \eqref{eq:g_col} are equal to \( 0 \).
\end{proof}

\begin{lem}
  Under the assumption in 
  Theorem~\ref{thm:main_g}, if \( \mu_{k}<\lambda_{k} \) and \( r_{k}>s_{k} \) for some \( k \), then the both sides of \eqref{eq:g_row} and \eqref{eq:g_col} are equal to \( 0 \).
\end{lem}
\begin{proof}
  Since \( \mu_{k}<\lambda_{k} \) and \( r_{k}>s_{k} \), we have \(
  \MRPP^{\row(\vr, \vs)}(\lm) = \MRPP^{\col(\vr, \vs)}(\lambda'/\mu') = \emptyset\) and
  therefore \( g_{\lambda/\mu}^{\row(\vr,\vs)}(\vx;\va,\vb) =
  g_{\lambda'/\mu'}^{\col(\vr,\vs)}(\vx;\va,\vb) = 0 \).

  We now show that the right hand side of \eqref{eq:g_row} and \eqref{eq:g_col}
  are equal to \( 0 \). By Lemma~\ref{lem:mu_subset_la} and
  Lemma~\ref{lem:det(h)=det*det}, we may assume \( \mu_{m}<\lambda_{m} \) for
  all \( m \). Let \( (R_{i,j})_{i,j=1}^n \) be the matrix in the right hand
  side of \eqref{eq:g_row}.

  We first suppose that \( \mu_{m}<\lambda_{m+1} \) for all \( m \). Then the
  assumption in Theorem~\ref{thm:main_g} gives \( r_{m}\le r_{m+1} \) and \(
  s_{m}\le s_{m+1} \) for all \( m \). Thus, for \( 1\le i\le k \) and \( k\le
  j\le n \), we have \( r_{j}\ge r_k>s_k\ge s_{i} \) and
  \( \mu_{j}\le \mu_k <\lambda_k\le \lambda_{i} \), which imply
  \begin{align}
    R_{i,j}
    &= h_{\lambda_i-\mu_j-i+j}[-A_{\lambda_i-1}+A_{\mu_j}+B_{i-1} - B_{j-1}] \nonumber \\
    &= h_{\lambda_i-\mu_j-i+j}[-(\alpha_{\mu_j+1}+\dots+\alpha_{\lambda_i-1})-(\beta_i+\dots+\beta_{j-1})] \nonumber \\
    \label{eq:R_ij=0}
    &= (-1)^{\lambda_i-\mu_j-i+j}
      e_{\lambda_i-\mu_j-i+j}(\alpha_{\mu_j+1},\dots,\alpha_{\lambda_i-1},\beta_i,\dots,\beta_{j-1}).
  \end{align}
  Since the number of variables in \eqref{eq:R_ij=0} is less than \(
  \lambda_{i}-\mu_{j}-i+j>0 \), we have \( R_{i,j}=0 \) for all \( 1\le i\le k
  \) and \( k\le j\le n \). This implies that the determinant in
  \eqref{eq:g_row} is \( 0 \).
  
  Now we suppose \( \mu_{m}\ge\lambda_{m+1} \) for some \( m \). For \( m+1\le
  i\le n \) and \(1\le j\le m \), we have \( \mu_j\ge \mu_m\ge \lambda_{m+1}\ge
  \lambda_i \) and therefore \( \lambda_i-\mu_j-i+j<0 \). Then, since \(
  R_{i,j}=h_{\lambda_{i}-\mu_{j}-i+j}[Z] \) for some formal power series \( Z
  \), we have \( R_{i,j}=0 \). Thus the matrix \( (R_{i,j})_{i,j=1}^{n} \) is a
  block upper triangular matrix and its determinant is written as
  \[
    \det(R_{i,j})_{i,j=1}^{n}=\det(R_{i,j})_{i,j=1}^{m}\det(R_{i,j})_{i,j=m+1}^{n}.
  \]
  Therefore, one can inductively deduce that the determinant in \eqref{eq:g_row}
  is equal to \( 0 \).
  
  Similarly, one can show that the right hand side of \eqref{eq:g_col} is equal to \( 0 \).
\end{proof}

The rest of this section is devoted to proving Theorem~\ref{thm:main_g}.

\subsection{Proof of \eqref{eq:g_row}: row-flagged dual Grothendieck polynomials}
\label{sec:proof-theor-refthm:m}

In this subsection we prove the first identity \eqref{eq:g_row} in
Theorem~\ref{thm:main_g}. Our strategy is as follows. We first show that
\eqref{eq:g_row} is equivalent to the following formula:
\begin{multline}
  \label{eq:g_row'}
  g_\lm^{\row(\vr, \vs)}(\vx;\va,\vb)\\
  = \det\left( \chi(r_j\le s_i) h_{\lambda_i-\mu_j-i+j}
    [X_{[r_j,s_i]}-A_{\lambda_i-1}+A_{\mu_j}+B_{i-1} - B_{j-1}]\right)_{i,j=1}^n.
\end{multline}
 We then find a generalization of this new formula and
prove it using recursions.

The following lemma shows that \eqref{eq:g_row} and \eqref{eq:g_row'} are equivalent.

\begin{lem}\label{lem:chi=no_chi}
  Let \( \lambda,\mu\in\Par_n, \vr,\vs\in \PP^n \) with \( \mu\subseteq\lambda
  \) and \( \vr\le\vs \) such that \( r_i\le r_{i+1} \) and \( s_i\le s_{i+1} \)
  whenever \( \mu_{i}<\lambda_{i+1} \). Then
 \begin{multline*}
   \det\left( \chi(r_j\le s_i)   h_{\lambda_i-\mu_j-i+j}
     [X_{[r_j,s_i]}-A_{\lambda_i-1}+A_{\mu_j}+B_{i-1} - B_{j-1}]
  \right)_{i,j=1}^n\\
  =\det\left(    h_{\lambda_i-\mu_j-i+j}
    [X_{[r_j,s_i]}-A_{\lambda_i-1}+A_{\mu_j}+B_{i-1} - B_{j-1}]
  \right)_{i,j=1}^n.
 \end{multline*}
\end{lem}
\begin{proof}
  Let \( L =(L_{i,j})_{i,j=1}^n \) and \( R =(R_{i,j})_{i,j=1}^n \) be the
  matrices in the left and right sides of the equation, respectively. If \( r_j
  \leq s_i\) for all \(i\) and \(j\), then the matrices \(L\) and \(R\) are the
  same, so it is done. Assume that \(L_{i,j} \neq R_{i,j}\) for some \( i \) and
  \( j \). In this case we must have \(r_j > s_i\) and \(L_{i,j}=0\). Thus it is
  enough to show that the \((i,j)\)-entry \(R_{i,j}\) does not contribute to the
  determinant of the matrix \(R\). We consider the two cases \( i>j \) and \(
    i<j \).
  
    First suppose \( i>j \). Since \(r_i \leq s_i < r_j\), by the condition that
    \( r_k\le r_{k+1} \) whenever \( \mu_{k}<\lambda_{k+1} \), there is at least
    one \(m \in \{j,\dots,i-1\}\) such that \( \mu_m \geq \lambda_{m+1}\). Then
    for integers \( p\) and \(q\) satisfying \(1 \leq p \leq m < q \leq n\), we
    have \(\lambda_q-\mu_p+p-q \leq \lambda_{m+1}-\mu_m+m-(m+1) <0\), which
    implies \( R_{p,q}=0 \). Therefore, \( \det R = \det (R_{k,l})_{k,l=1}^m \det
    (R_{k,l})_{k,l=m+1}^n\). Since \( j\le m<i \), this means that the
    \((i,j)\)-entry \(R_{i,j}\) does not contribute to the determinant of the
    matrix \(R\).

    We now suppose \(i<j\). If \(\mu_j < \lambda_i\), then by the same argument in \eqref{eq:R_ij=0}, we have \( R_{i,j}=0 \). This is a contradiction to the assumption \(L_{i,j}\neq R_{i,j}\), so we must have
  \(\lambda_i \leq \mu_j\). Then, since \( \mu_i \leq \lambda_i \leq \mu_j \leq \mu_i
  \), we also have \(\mu_i=\lambda_i\).
  Thus by~\Cref{lem:det(h)=det*det}, we obtain  \(\det R =\det(R_{k,l})_{k,l=1}^{i-1}
  \det(R_{k,l})_{k,l=i+1}^n\). This means that the \((i,j)\)-entry of
  \(R\) does not contribute to the determinant of the matrix
  \(R\), which completes the proof.
\end{proof}

Our next step is to find a generalization of \eqref{eq:g_row'}. To do this we
need the following definitions.

\begin{defn}
  A \emph{dented partition} is a sequence \( (\lambda_1,\dots,\lambda_n) \) of
  nonnegative integers such that 
  \[
    \lambda_1+1=\dots=\lambda_{k-1}+1=\lambda_k\ge \dots\ge \lambda_n,
  \]
  for some \( 1\le k\le n \). The cell \( (k,\lambda_k) \) is called the
  \emph{minimal cell} of \( \lambda \). Similar to partitions, each \(
  \lambda_i \) of a dented partition \( \lambda=(\lambda_1,\dots,\lambda_n) \)
  is called a \emph{part}. The \emph{Young diagram} of \( \lambda \) is also
  defined as usual. We denote by \( \DPar_n \) the set of all dented partitions
  with at most \( n \) parts. See Figure~\ref{fig:ydp}.
\end{defn}

\begin{figure}
\ytableausetup{boxsize=1.5em}
    \begin{ytableau}
       ~ & ~ & ~ \\ 
       ~ & ~ & ~ \\
       ~ & ~ & ~ & *(gray) \\
       ~ & ~ & ~& ~  \\
       ~
    \end{ytableau}
\caption{The Young diagram of a dented partition \( \lambda=(3,3,4,4,1) \).
 The cell \( (3,4) \) is the minimal cell, which is colored gray.}
\label{fig:ydp}
\end{figure}

Note that a dented partition \( \lambda \)
is a partition if and only if its minimal cell is \( (1,\lambda_1) \). 

\begin{remark}
  The motivation of introducing dented partitions is to construct a marked RPP
  of shape \( \lambda \) by filling a cell one at a time with respect to the
  total order \( \prec \) on the cells defined by \( (i,j) \prec (i',j') \) if
  \( j>j' \) or \( j=j' \) and \( i\le i' \). At each stage of the construction,
  the remaining cells form a dented partition and the next cell to be filled is
  the minimal cell of this dented partition.

  We also note that dented partitions are a special case of
  pseudo-partitions introduced by Aas, Grinberg, and Scrimshaw, who found a
  Jacobi--Trudi-like formula for the generating function for semistandard Young
  tableaux of a pseudo-partition shape with a flag condition in
  \cite[Theorem~5.3]{Aas2020}.
\end{remark}

\begin{defn}\label{defn:MRPPI}
  Let \( \lambda \in \DPar_n \) and \( \mu\in \Par_n \) with \(
  \mu\subseteq\lambda \), and let \( \vr,\vs\in\PP^n \) and \( I\subseteq [n] \). If
  \( \vr\not\le \vs \), we define \( \MRPP^{\row(\vr, \vs)}_I(\lambda/\mu)=\emptyset
  \). If \( \vr\le \vs \), we define \( \MRPP^{\row(\vr, \vs)}_I(\lm) \) to be
  the set of RPPs \( T \) of shape \( \lm \)  satisfying the following conditions.
    \begin{itemize}
    \item For all \( (i,j)\in\lm \), we have \( r_i \le T(i,j) \le s_i \).
\item If \( (k, \lambda_k) \) is the minimal cell of \( \lambda \) and \( k-1\in I \),
  then \( T(k,\lambda_k)\ge s_{k-1} \).
    \item Every entry $T(i,j)$ with $T(i,j)=T(i,j+1)$ may be marked, where for each \( i\in [n] \) we define
  \[
    T(i,\lambda_i+1) =
    \begin{cases}
      \infty & \mbox{if \( i\not\in I \)},\\
      s_i & \mbox{if \( i\in I \).}
    \end{cases}
  \]
    \end{itemize}
For \( T\in \MRPP^{\row(\vr,\vs)}_I(\lm) \), we define
\[
  \wt(T) = \prod_{(i,j)\in\lm} \wt(T(i,j)),
\]
where
\begin{align*}
  \wt(T(i,j)) =
  \begin{cases}
    -\alpha_{j} & \mbox{if $T(i,j)$ is marked,}\\
    \beta_{i-1} & \mbox{if $T(i,j)$ is not marked and $T(i,j)= T(i-1,j)$,}\\
    x_{T(i,j)} & \mbox{otherwise.}
  \end{cases}
\end{align*}
Here, the equality $T(i,j)= T(i-1,j)$ means that their underlying integers are equal.
\end{defn}

Note that in the above definition the assumption \( T(i,\lambda_i+1)=s_i \) for
\( i\in I \) also plays a role in the definition of \( \wt(T) \). 
See Figure~\ref{fig:ydpi}. Note also that
\( \MRPP^{\row(\vr, \vs)}_I(\lm)=\emptyset \) if \( r_i>s_i \) for
some \( i \) (even for the case \( \mu_i=\lambda_i \)). Thus we have
\[
  \MRPP^{\row(\vr, \vs)}_I(\mu/\mu)=
  \begin{cases}
    \emptyset  & \mbox{if \( \vr\not\le \vs \)},\\
    \{\emptyset\}  & \mbox{otherwise.}
  \end{cases}
\]

\begin{figure}
  \ytableausetup{mathmode, boxsize=2em}
    \begin{ytableau}
      \none[1\le] & \none[~] & 1 & 2 & \none[~] & \none[\le3] \\ 
      \none[1\le] & \none[~] & 1 & 3^* & 3 & \none[\le3] \\ 
      \none[2\le] & 2 & 2 & 3 & 4^*  & \none[\le4] \\
      \none[2\le] & 4 & \none[~] & \none[~] & \none[~] & \none[\le5] 
    \end{ytableau}
\caption{An example of  \( T\in\MRPP_I^{\row(\vr,\vs)}(\lm) \),
where \( \lambda=(3,4,4,1),\mu=(1,1),I=\{ 1,3 \}, \vr=(1,1,2,2), \) 
and \( \vs=(3,3,4,5) \). Note that 
\( T(2,4)=3 \) cannot be marked because \( 2\notin I \) and its contribution
to the weight is \( \beta_1 \) since \( 1\in I \) and \( T(2,4)=T(1,4)=3 \). Thus we have
\( \wt(T)= x_1x_2^3x_4(-\alpha_3)(-\alpha_4)\beta_1^2\beta_2 \). }
\label{fig:ydpi}
\end{figure}

The following proposition is the promised generalization of \eqref{eq:g_row'}.
 
\begin{prop}\label{prop:MRPPI}
  Let \( \lambda\in\DPar_n,\mu\in\Par_n, \vr,\vs\in \PP^n \) with \(
  \mu\subseteq\lambda \) such that \( r_i\le r_{i+1} \) and \( s_i\le s_{i+1} \)
  whenever \( \mu_{i}<\lambda_{i+1} \). Suppose that \( (k,\lambda_k) \) is the
  minimal cell of \( \lambda \) and \( I = \{1,2,\dots,p\} \) or \( I =
  \{1,2,\dots,p\}\setminus\{k\} \) for some integer \( k\le p\le \ell(\lambda)
  \) satisfying
  \[
    \lambda_1+1=\dots=\lambda_{k-1}+1=\lambda_k=\lambda_{k+1}=\dots=\lambda_p.
  \]
  Then
  \begin{equation}\label{eq:refined_conj}
    \sum_{T\in\MRPP^{\row(\vr, \vs)}_I(\lm)} \wt(T) = 
    \det\left( h_{\lambda,\mu}^{\vr,\vs}(I;i,j)  \right)_{i,j=1}^n,
  \end{equation} 
  where
  \[
      h_{\lambda,\mu}^{\vr,\vs}(I;i,j) =
      \chi(r_j\le s_i)   h_{\lambda_i-\mu_j-i+j}
      [X_{[r_j,s_i]}-A_{\lambda_i-1+\chi(i\in I)}+A_{\mu_j}+B_{i-1} -
      B_{j-1}].
    \]
\end{prop}

Observe that \eqref{eq:g_row'} is the special case, \( \lambda\in\Par_n \) and
\( I=\emptyset \), of Proposition~\ref{prop:MRPPI}. For the rest of this
subsection we prove this proposition by finding the same recursions and initial
conditions for both sides of \eqref{eq:refined_conj}. Let
\begin{align*}
  M(\lambda,\mu,\vr,\vs,I) &= \sum_{T\in\MRPP^{\row(\vr, \vs)}_I(\lm)} \wt(T),\\
  D(\lambda,\mu,\vr,\vs,I) &=\det\left( h_{\lambda,\mu}^{\vr,\vs}(I;i,j)  \right)_{i,j=1}^n.
\end{align*}

The following two lemmas give the same initial conditions for \(
M(\lambda,\mu,\vr,\vs,I) \) and \( D(\lambda,\mu,\vr,\vs,I) \).

\begin{lem}\label{lem:LHS=RHS=0}
  Let \( \lambda\in\DPar_n,\mu\in\Par_n,  \vr, \vs\in\RPar_n\) with \( \mu\subseteq\lambda \).
  If \( \vr\not\le\vs \), then
  \[
    M(\lambda,\mu,\vr,\vs,I) = D(\lambda,\mu,\vr,\vs,I)= 0.
  \]  
\end{lem}
\begin{proof}
  By assumption there is an integer \( 1\le k\le n \) with \( r_k>s_k \). By
  definition we have \( M(\lambda,\mu,\vr,\vs,I) =0 \). To show \(
  D(\lambda,\mu,\vr,\vs,I) =0 \) it suffices to show that \(
  h_{\lambda,\mu}^{\vr,\vs}(I;i,j)=0 \) for all \( 1\le i\le k \) and \( k\le
  j\le n \). Since \( \vr,\vs\in\RPar_n \), we have \( r_j\ge r_k>s_k\ge s_i \),
  and therefore \( \chi(r_j\le s_i)=0 \). This implies \(
  h_{\lambda,\mu}^{\vr,\vs}(I;i,j)=0 \).
\end{proof}

\begin{lem}\label{lem:LHS=RHS=1}
Let \( \mu\in\Par_n, \vr, \vs\in \PP^n\).
 If \( \vr\le \vs \), then
  \[
    M(\mu,\mu,\vr,\vs,I) =D(\mu,\mu,\vr,\vs,I)= 1.
  \]
  
\end{lem}
\begin{proof}
  By definition we have \( M(\mu,\mu,\vr,\vs,I) =1 \). The other equation \(
  D(\mu,\mu,\vr,\vs,I) =1 \) can easily be proved by observing that the
  matrix in the definition \( D(\mu,\mu,\vr,\vs,I) =\det
  (h_{\mu,\mu}^{\vr,\vs}(I;i,j)) \) is upper unitriangular.
\end{proof}

The following two lemmas give the same recursions for \(
M(\lambda,\mu,\vr,\vs,I) \) and \( D(\lambda,\mu,\vr,\vs,I) \).

\begin{lem}\label{lem:rec1-dual}
  Let \( \lambda\in\DPar_n,\mu\in\Par_n,\vr,\vs\in\PP^n \) with \( \mu\subseteq\lambda \).
  Suppose that $1\le t\le n-1$ is an integer such that $\mu_t\ge\lambda_{t+1}$.
  Let \( \phi_{\vb} \) be the map defined by shifting \( \beta_i \) to \( \beta_{i+1}
  \) for \( i\ge1 \). Then 
  for \( F\in\{D,M\} \), we have
  \[
    F(\lambda,\mu,\vr,\vs,I)
    = F(\lambda^{(1)},\mu^{(1)},\vr^{(1)},\vs^{(1)},I^{(1)})
    \phi_{\vb}^t(F(\lambda^{(2)},\mu^{(2)},\vr^{(2)},\vs^{(2)},I^{(2)})),
  \]
    where \( \gamma^{(1)} = (\gamma_1,\dots,\gamma_t) \) and
    \( \gamma^{(2)} = (\gamma_{t+1},\dots,\gamma_n) \)
    for each \( \gamma\in \{\lambda,\mu,\vr,\vs\} \),
    and \( I^{(1)}= I\cap\{1,2,\dots,t\} \)
    and \( I^{(2)}= I\cap\{t+1,t+2,\dots,n\} \).
\end{lem}
\begin{proof}
  If \( F=M \), the statement follow immediately from the definition
  of \( M(\lambda,\mu,\vr,\vs,I) \). For the case \( F=D \), since 
  \begin{align*}
    D(\lambda^{(1)},\mu^{(1)},\vr^{(1)},\vs^{(1)},I^{(1)}) &=\det\left( h_{\lambda,\mu}^{\vr,\vs}(I;i,j)  \right)_{i,j=1}^t,\\
    \phi_{\vb}^t(D(\lambda^{(2)},\mu^{(2)},\vr^{(2)},\vs^{(2)},I^{(2)}))
    &=\det\left( h_{\lambda,\mu}^{\vr,\vs}(I;i,j)  \right)_{i,j=t+1}^n,
  \end{align*}
  it suffices to show that \(
  h_{\lambda,\mu}^{\vr,\vs}(I;i,j)=0 \) for all \( t+1\le i\le n \) and \( 1\le
  j\le t \).
  Suppose \( 1\le j\le t \).
  If \( i=t+1 \), then 
  \[
    \lambda_i-i-\mu_j+j \le \lambda_{t+1}-(t+1)-\mu_t+t\le -1,
  \]
  and therefore \(h_{\lambda,\mu}^{\vr,\vs}(I;i,j)=0 \). If \( t+2\le i\le n \),
  since \( \lambda_b\le \lambda_a+1 \) for any \( 1\le a<b\le n \), we have
  \[
    \lambda_i-i-\mu_j+j \le (\lambda_{t+1}+1)-(t+2)-\mu_t+t\le-1.
  \]
  Thus \(h_{\lambda,\mu}^{\vr,\vs}(I;i,j)=0 \) and we obtain the desired identity.
\end{proof}

\begin{lem}\label{lem:rec_main1}
  Let \( \lambda\in\DPar_n \) be a dented partition with the minimal cell
  \( (k,\lambda_k) \) and let \( \mu\in\Par_n \) be a partition with
  \( \mu\subseteq\lambda\setminus\{(k,\lambda_k)\} \).
  Let \( \vr,\vs\in \RPar_n \) with \( \vr\le\vs \).
  Suppose that \( I = \{1,2,\dots,p\} \) or \( I =
  \{1,2,\dots,p\}\setminus\{k\} \) for some integer \( k\le p\le \ell(\lambda)
  \) satisfying
  \[
    \lambda_1+1=\dots=\lambda_{k-1}+1=\lambda_k=\lambda_{k+1}=\dots=\lambda_p.
  \]
   If \( \lambda \) is not a partition and \( s_{k-1}=s_k \), then
   for \( F\in\{D,M\} \),
   \[
     F(\lambda,\mu,\vr,\vs,I)
     =(\beta_{k-1}-\chi(k\in I) \alpha_{\lambda_k}) F(\lambda-\epsilon_k,\mu,\vr,\vs,I\cup\{k\}).
   \]
   If \( \lambda \) is a partition or \( s_{k-1}<s_k \), then
   for \( F\in\{D,M\} \),
   \[
     F(\lambda,\mu,\vr,\vs,I)
     =(x_{s_k}-\chi(k\in I) \alpha_{\lambda_k}) F(\lambda-\epsilon_k,\mu,\vr,\vs,I\cup\{k\})
     +F(\lambda,\mu,\vr,\vs-\epsilon_k,I\setminus\{k\}).
   \]
 \end{lem}
\begin{proof}
  We first consider the case \( F=M \). For the first statement, suppose that \( \lambda \) is
  not a partition and \( s_{k-1}=s_k \). Let
  \( T\in\MRPP^{\row(\vr, \vs)}_I(\lm) \). Note that since
  \( k-1\in I \), we have \( T(k-1,\lambda_{k-1}+1) = s_{k-1} \) and
  \( s_{k-1}\le T(k,\lambda_k)\le s_k \). Thus
  \( T(k-1,\lambda_{k}) = T(k-1,\lambda_{k-1}+1) = s_{k-1}=s_k \) and
  \( T(k,\lambda_k)=s_k \). This implies that
  \( \wt(T(k,\lambda_k)) \) is \( -\alpha_{\lambda_k} \) or \( \beta_{k-1} \) if
  \( k\in I \) and \( \beta_{k-1} \) otherwise. Let \( T' \) be the
  tableau obtained from \( T \) by removing \( T(k,\lambda_k) \). Then
  \( T\mapsto T' \) is a bijection from \( M(\lambda,\mu,\vr,\vs,I) \)
  to \( M(\lambda-\epsilon_k,\mu,\vr,\vs,I\cup \{k\}) \) such that
  \( \wt(T)=\wt(T')\cdot (\beta_{k-1}-\chi(k\in I) \alpha_{\lambda_k})
  \). Therefore,
   \[
     M(\lambda,\mu,\vr,\vs,I)
     =(\beta_{k-1}-\chi(k\in I) \alpha_{\lambda_k}) M(\lambda-\epsilon_k,\mu,\vr,\vs,I\cup\{k\}),
   \]
   which is the first statement. The second statement can be proved
   similarly with two cases \( T(k,\lambda_k)=s_k \) and
   \( T(k,\lambda_k)<s_k \), which give
   \( (x_{s_k}-\chi(k\in I) \alpha_{\lambda_k}) F(\lambda-\epsilon_k,\mu,\vr,\vs,I\cup\{k\}) \)
   and \( F(\lambda,\mu,\vr,\vs-\epsilon_k,I\setminus\{k\}) \), respectively.

   \medskip
Now we consider the case \( F=D \). Let
  \[
    H = \left( h_{\lambda,\mu}^{\vr,\vs}(I;i,j)  \right)_{i,j=1}^n.
  \]

  For the first statement suppose that \( \lambda \) is not a partition and \(
  s_{k-1}=s_k \). We claim that for all \( 1\le j\le n \),
  \begin{equation}\label{eq:hIkj-claim}
    h_{\lambda,\mu}^{\vr,\vs}(I;k,j) = h_{\lambda,\mu}^{\vr,\vs}(I;k-1,j) 
    + (\beta_{k-1}-\chi(k\in I) \alpha_{\lambda_k})h_{\lambda-\epsilon_k,\mu}^{\vr,\vs}(I\cup\{k\};k,j). 
  \end{equation}
  Note that by subtracting the \( (k-1) \)st row from the \( k \)th row in \( H
  \), the claim implies the first statement. Thus it suffices to prove
  \eqref{eq:hIkj-claim}.
  
  For brevity, let 
  \begin{equation}
    \label{eq:mZ}
     m=\lambda_k-\mu_j+j-k, \qquad
    Z=X_{[r_j,s_k]}-A_{\lambda_k-1}+A_{\mu_j}+B_{k-1}-B_{j-1}.     
  \end{equation}
Then
\begin{align}
  \notag
  h_{\lambda,\mu}^{\vr,\vs}(I;k,j)
  &= \chi(r_j\le s_k)  h_{\lambda_k-\mu_j+j-k}
  [X_{[r_j,s_k]}-A_{\lambda_k-1+\chi(k\in I)}+A_{\mu_j}+B_{k-1} - B_{j-1}]\\
  \label{eq:hIkj}
  &= \chi(r_j\le s_k) h_m[Z-\chi(k\in I)\alpha_{\lambda_k}],
\end{align}
where the second equality follows from \( A_{\lambda_k-1+\chi(k\in I)} =
A_{\lambda_k-1}+\chi(k\in I)\alpha_{\lambda_k}\).
By Lemma~\ref{lem:h_m[Z]},
\begin{align}
  \notag
  h_m[Z-\chi(k\in I)\alpha_{\lambda_k}]
  &= h_m[Z] - \chi(k\in I)\alpha_{\lambda_k}h_{m-1}[Z]\\
  \label{eq:hIkj'}
  &= h_m[Z-\beta_{k-1}] +\beta_{k-1}h_{m-1}[Z]- \chi(k\in I)\alpha_{\lambda_k}h_{m-1}[Z].
\end{align}
On the other hand, since \( \lambda_{k-1}=\lambda_k-1 \) and \( k-1\in I \), 
  \begin{align}
    \notag
    h_{\lambda,\mu}^{\vr,\vs}(I;k-1,j)
    \notag
    &= \chi(r_j\le s_{k-1})  h_{\lambda_{k-1}-\mu_j+j-(k-1)}
      [X_{[r_j,s_{k-1}]}-A_{\lambda_{k-1}}+A_{\mu_j}+B_{k-2} - B_{j-1}]\\
    \notag
    &= \chi(r_j\le s_k)  h_{\lambda_k-\mu_j+j-k}
      [X_{[r_j,s_k]}-A_{\lambda_k-1}+A_{\mu_j}+B_{k-2} - B_{j-1}]\\
    \label{eq:hIk-1j}
    &= \chi(r_j\le s_k)  h_{m}[Z-\beta_{k-1}].
  \end{align}
  Since
  \( h_{\lambda-\epsilon_k,\mu}^{\vr,\vs}(I\cup\{k\};k,j) =  \chi(r_j\le s_k)  h_{m-1}[Z] \),
  \eqref{eq:hIkj}, \eqref{eq:hIkj'}, and \eqref{eq:hIk-1j} give the claim
  \eqref{eq:hIkj-claim}. This shows the first statement.

  For the second statement suppose that \( \lambda \) is a partition or \( s_{k-1}<s_k \).
  We claim that for all \( 1\le j\le n \),
  \begin{equation}\label{eq:hIkj-claim2}
    h_{\lambda,\mu}^{\vr,\vs}(I;k,j) = h_{\lambda,\mu}^{\vr,\vs-\epsilon_k}(I\setminus\{k\};k,j)
      + (x_{s_k}-\chi(k\in I)\alpha_{\lambda_k}) h_{\lambda-\epsilon_k,\mu}^{\vr,\vs}(I\cup\{k\};k,j).
  \end{equation}
  Note that by the linearity of the determinant in the \( k \)th row of \( H \), the claim implies
  the second statement. Thus it suffices to prove \eqref{eq:hIkj-claim2}.

  We again use the notation \( m \) and \( Z \) in \eqref{eq:mZ}. By the same
  computation as in \eqref{eq:hIkj'} with \( \beta_{k-1} \) replaced by \(
  x_{s_k} \), we obtain
  \begin{align*}
   h_{\lambda,\mu}^{\vr,\vs}(I;k,j) 
   &= \chi(r_j\le s_k)h_m[Z-x_{s_k}]+\chi(r_j\le s_k)(x_{s_k}-\chi(k\in I)\alpha_{\lambda_k})h_{m-1}[Z]\\
   \label{eq:hIkj-2}&= \chi(r_j\le s_k)h_m[Z-x_{s_k}]+(x_{s_k}-\chi(k\in I)\alpha_{\lambda_k})
   h_{\lambda-\epsilon_k,\mu}^{\vr,\vs}(I\cup\{k\};k,j).
  \end{align*}
  Since \( h_{\lambda,\mu}^{\vr,\vs-\epsilon_k}(I\setminus\{k\};k,j) =
  \chi(r_j\le s_k-1) h_m[Z-x_{s_k}] \), by the above equation, in order to prove
  \eqref{eq:hIkj-claim2} it suffices show that
  \begin{equation*}
    \chi(r_j\le s_k)h_m[Z-x_{s_k}]= \chi(r_j\le s_k-1)h_m[Z-x_{s_k}].
  \end{equation*}
  If \( r_j\ne s_k \), then \( \chi(r_j\le s_k) = \chi(r_j\le s_k-1) \) and we
  are done. Thus we may assume \( r_j=s_k \). In this case we must show that \(
  h_m[Z-x_{s_k}]=0 \), i.e.,
  \begin{equation}\label{eq:chi-sk-1}
    h_{\lambda_k-\mu_j+j-k}[-A_{\lambda_k-1}+A_{\mu_j}+B_{k-1}-B_{j-1}] = 0.
  \end{equation}
  To prove \eqref{eq:chi-sk-1} we consider the two cases
  \( j<k \) and \( j\ge k \) .

  Suppose \( j<k \). In this case we have \( k\ge 2 \). 
  Since \( r_{k-1}\le
  s_{k-1}<s_k=r_j \) and \(\vr \in \RPar_n\), we have \( k-1 < j \), 
  which is a contradiction to \( j<k \).
  
  We now suppose \( j\ge k \). If \( k=1 \), then by assumption we have \(
  \lambda_k-1\ge\mu_k\ge\mu_j \). If \( k\ge2 \), since \( \lambda \) is a
  dented partition and \( \mu \) is a partition contained in \( \lambda \), we
  have \( \lambda_k-1=\lambda_{k-1}\ge\mu_{k-1}\ge\mu_j \). Therefore we always
  have \(\lambda_k-1\ge\mu_j\) and the left side of \eqref{eq:chi-sk-1} is equal
  to
  \begin{multline*}
    (-1)^{\lambda_k-\mu_j+j-k} e_{\lambda_k-\mu_j+j-k}[A_{\lambda_k-1}-A_{\mu_j}-B_{k-1}+B_{j-1}]\\
    =(-1)^{\lambda_k-\mu_j+j-k} e_{\lambda_k-\mu_j+j-k}(\alpha_{\mu_j+1},\dots,\alpha_{\lambda_k-1},\beta_k,\dots,\beta_{j-1}),
  \end{multline*}
  which is equal to \( 0 \) because the number \( \lambda_k-1-\mu_j+(j-1)-(k-1)
  \) of variables is less than \( \lambda_k-\mu_j+j-k \).
\end{proof}

Now we are ready to prove Proposition~\ref{prop:MRPPI}.

\begin{proof}[Proof of Proposition~\ref{prop:MRPPI}]
  We proceed by induction on \( (M,N) \), where
  \( M:=\sum_{i=1}^n(s_i-r_i) \) and \( N:=|\lm| \). There are two
  base cases: the case \( M < 0 \), which follows from
  Lemmas~\ref{lem:LHS=RHS=0} and \ref{lem:rec1-dual}, and the case
  \( N=0 \), which follows from Lemma~\ref{lem:LHS=RHS=1}.

For the inductive step, suppose that \( M\ge 0, N\ge1 \) and the result holds
for all pairs \( (M',N') \) with \( M'<M \) or \( N'<N \). Now consider a tuple
\( (\lambda,\mu,\vr,\vs,I) \) satisfying the conditions in
Proposition~\ref{prop:MRPPI} such that \( \sum_{i=1}^n(s_i-r_i)=M \) and \(
|\lm|=N \). If \( \vr\not\le\vs \), then the result holds by
Lemma~\ref{lem:LHS=RHS=0}. Therefore we may assume \( \vr\le\vs \).
We consider the following two cases.

\textbf{Case 1:} Suppose that \( \mu_t<\lambda_{t+1} \) for all \( 1\le t\le n-1
\). By the assumptions on \( \vr \) and \( \vs \), in this case we have \(
\vr,\vs\in\RPar_n \). Let \( (k,\lambda_k) \) be the minimal cell of \( \lambda
\). Since \( \mu_k<\lambda_{k+1}\le \lambda_k \), we have \( (k,\lambda_k)\in\lm
\). Hence by Lemma~\ref{lem:rec_main1} we can
decrease either \( M=\sum(s_i-r_i) \) or \( N=|\lm| \). Thus this case is obtained
by the induction hypothesis.

\textbf{Case 2:} 
Suppose that an integer \( 1\le t\le n-1 \) satisfies \(
\mu_t\ge\lambda_{t+1} \).
By Lemma~\ref{lem:rec1-dual}, 
it suffices to show that
\begin{align*}
  M(\lambda^{(1)},\mu^{(1)},\vr^{(1)},\vs^{(1)},I^{(1)})
  &= D(\lambda^{(1)},\mu^{(1)},\vr^{(1)},\vs^{(1)},I^{(1)}) \mbox{ and } \\
  M(\lambda^{(2)},\mu^{(2)},\vr^{(2)},\vs^{(2)},I^{(2)})
  &= D(\lambda^{(2)},\mu^{(2)},\vr^{(2)},\vs^{(2)},I^{(2)}),
\end{align*}
where we follow the notations in Lemma~\ref{lem:rec1-dual}.
In this way we can reduce the size of \( n \) without increasing \( M \) or \( N
\). Observe that if \( n=1 \) we must have Case 1 above. Therefore by repeating
this process we can divide the problem of proving the result for the tuple \(
(\lambda,\mu,\vr,\vs,I)\) into those for some tuples \(
(\lambda^*,\mu^*,\vr^*,\vs^*,I^*)\) with \( M^*=\sum (s^*_i-r^*_i)\le M \) and \(
N^*=|\lambda^*/\mu^*|\le N \) that are all in Case 1. Then this case is obtained
by the induction hypothesis and by Case 1.
\end{proof}

\subsection{Proof of \eqref{eq:g_col}: column-flagged dual Grothendieck polynomials}

In this subsection we prove the second assertion \eqref{eq:g_col} of
Theorem~\ref{thm:main_g}. Our strategy is similar to that in the previous
subsection.
We first show that \eqref{eq:g_col} is equivalent to
\begin{multline}
  \label{eq:g_col'}
  g_{\lambda'/\mu'}^{\col(\vr, \vs)}(\vx;\va,\vb)\\
  = \det\left( \chi(r_j\le s_i)   e_{\lambda_i-\mu_j-i+j}
    [X_{[r_j,s_i]}-A_{i-1}+A_{j-1}+B_{\lambda_i-1} - B_{\mu_j}]
  \right)_{i,j=1}^n.
\end{multline}
We then find a generalization of this formula and prove it using recursions.

The following lemma shows that \eqref{eq:g_col} and \eqref{eq:g_col'} are equivalent.

\begin{lem}\label{lem:chi=no_chi_e}
  Let \( \lambda,\mu\in\Par_n, \vr,\vs\in \PP^n \) with \( \mu\subseteq\lambda
  \) and \( \vr\le\vs \) such that \( r_i\le r_{i+1} \) and \( s_i\le s_{i+1} \)
  whenever \( \mu_{i}<\lambda_{i+1} \). Then
  \begin{multline*}
    \det\left( \chi(r_j\le s_i)   e_{\lambda_i-\mu_j-i+j}
    [X_{[r_j,s_i]}-A_{i-1}+A_{j-1}+B_{\lambda_i-1} - B_{\mu_j}]
    \right)_{i,j=1}^n \\
    =\det\left(    e_{\lambda_i-\mu_j-i+j}
    [X_{[r_j,s_i]}-A_{i-1}+A_{j-1}+B_{\lambda_i-1} - B_{\mu_j}]
    \right)_{i,j=1}^n.
  \end{multline*}
\end{lem}
\begin{proof}
  This can be proved by the same argument in the proof of
  Lemma~\ref{lem:chi=no_chi}. The only difference is the following analogue of
  \eqref{eq:R_ij=0}: if \(i<j\) and \(\mu_j < \lambda_i\), then we have
  \[
    e_{\lambda_i-\mu_j-i+j}[-A_{i-1}+A_{j-1}+B_{\lambda_i-1} - B_{\mu_j}]
    = e_{\lambda_i-\mu_j-i+j}
    (\alpha_i,\dots,\alpha_{j-1}, \beta_{\mu_j+1},\dots,\beta_{\lambda_i-1})
    =0.
  \]
  Then the argument in the proof of Lemma~\ref{lem:chi=no_chi} with this fact completes the proof.
\end{proof}

Our next step is to find a generalization of \eqref{eq:g_col'}. We first
reformulate \eqref{eq:g_col'}. To do this we introduce two different versions of
marked RPPs.

\begin{defn}
  A \emph{right-marked RPP} of shape $\lm$ is an RPP $T$ of shape $\lm$ in which every
  entry $T(i,j)$ with $T(i,j)=T(i,j-1)$ may be marked. We denote by
  $\RMRPP(\lm)$ the set of right-marked RPPs of shape $\lm$.
  For $T\in\RMRPP(\lm)$, define
  \[
    \wt_R(T) = \prod_{(i,j)\in\lm} \wt_R (T(i,j)),
  \]
  where
  \begin{align*}
    \wt_R(T(i,j)) =
    \begin{cases}
      -\alpha_{j-1} & \mbox{if $T(i,j)$ is marked,}\\
      \beta_{i} & \mbox{if $T(i,j)$ is not marked and $T(i,j)= T(i+1,j)$,}\\
      x_{T(i,j)} & \mbox{otherwise.}
    \end{cases}
  \end{align*}
  See Figure~\ref{fig:MRPP}.
\end{defn}

\begin{figure}
  \ytableausetup{mathmode, boxsize=1.7em}
  \ytableaushort{\none\none124{4^*},\none1{1^*}35,\none11,3{3^*}{3^*}}
\caption{An example of \( T\in\RMRPP(\lm)\), where \( \lambda=(6,5,3,3) \) and 
\( \mu=(2,1,1) \). We have the weight \( \wt(T)=x_1^2x_2x_3^2x_4x_5
(-\alpha_1)(-\alpha_2)^2(-\alpha_5)\beta_1\beta_2 \). }
  \label{fig:MRPP}
\end{figure}

\begin{defn}
  A \emph{bottom-marked RPP} of shape $\lm$ is an RPP $T$ of shape $\lm$ in which every
  entry $T(i,j)$ with $T(i,j)=T(i-1,j)$ may be marked. We denote by
  $\BMRPP(\lm)$ the set of bottom-marked RPPs of shape $\lm$.
  For $T\in\BMRPP(\lm)$, define
  \[
    \wt_B(T) = \prod_{(i,j)\in\lm} \wt_B (T(i,j)),
  \]
  where
  \begin{align*}
    \wt_B(T(i,j)) =
    \begin{cases}
      -\alpha_{i-1} & \mbox{if $T(i,j)$ is marked,}\\
      \beta_{j} & \mbox{if $T(i,j)$ is not marked and $T(i,j)= T(i,j+1)$,}\\
      x_{T(i,j)} & \mbox{otherwise.}
    \end{cases}
  \end{align*}
\end{defn}
One can easily see that bottom-marked RPPs are just the transposes of
right-marked RPPs, and their weights coincide, that is,
for \( T\in\RMRPP(\lambda/\mu) \), \( \wt_R(T) = \wt_B(T') \).

For an RPP \( T_0 \) of shape \( \lm \), let \( \MRPP(T_0) \) (resp.~\( \RMRPP(T_0) \)) be the set of left-marked (resp.~right-marked) RPPs \( T \) whose underlying RPP is \( T_0 \). To avoid confusion, in this section we also write \( \wt_L \) for the weight function on left-marked RPPs given in Definition~\ref{def:left MRPP}.
\begin{prop} \label{prop:leftmarked=rightmarked}
  For an RPP \( T_0 \) of shape \( \lm \),
\[
  \sum_{T\in\MRPP(T_0)} \wt_L(T) = \sum_{T\in\RMRPP(T_0)} \wt_R(T).
\]
\end{prop}
\begin{proof}
  For \( k\ge 1 \), let \( T_0[k] \) be the RPP obtained from \( T_0 \) by taking the cells filled with \( k \). Then, by definition,
  \[
    \sum_{T\in\MRPP(T_0)} \wt_L(T) = \prod_{k\ge 1} \sum_{T\in\MRPP(T_0[k])}\wt_L(T)
  \]
  and
  \[
    \sum_{T\in\RMRPP(T_0)} \wt_R(T) = \prod_{k\ge 1} \sum_{T\in\RMRPP(T_0[k])}\wt_R(T).
  \]
  Therefore, without loss of generality we may assume that every entry in \( T_0 \) is \( k \) for some fixed \( k\ge 1 \).
  Then, for a left-marked RPP \( T\in\MRPP(T_0) \), we have
  \begin{equation}  \label{eq:wt_left}
    \wt_L(T) = x_k^{m_L} \prod_{j\ge 1} (-\alpha_j)^{a_{L,j}} \prod_{i\ge 1} \beta_i^{b_{L,i}},
  \end{equation}
  where \( m_L \) is the number of cells \( u \) with \( k \) in \( T \) such that
  the cell above \( u \) is not in \( \lm \), \( a_{L,j} \) is the number of
  cells with \( k^* \) in the \( j \)th column of \( T \), and \( b_{L,i} \) is the
  number of cells \( u \) with \( k \) in the \( (i+1) \)st row of \( T \) such that
  the cell above \( u \) is in \( \lm \). Similarly, for a right-marked RPP \(
  T\in\RMRPP(T_0) \),
  \begin{equation}  \label{eq:wt_right}
    \wt_R(T) = x_k^{m_R} \prod_{j\ge 1} (-\alpha_j)^{a_{R,j}} \prod_{i\ge 1} \beta_i^{b_{R,i}},
  \end{equation}
  where \( m_R \) is the number of cells \( u \) with \( k \) in \( T \) such that
  the cell below \( u \) is not in \( \lm \), \( a_{R,j} \) is the number of
  cells with \( k^* \) in the \( (j+1) \)st column of \( T \), and \( b_{R,i} \) is the
  number of cells \( u \) with \( k \) in the \( i \)th row of \( T \) such that the
  cell below \( u \) is in \( \lm \).

  For a left-marked RPP \( T\in\MRPP(T_0) \), let \( \varphi(T) \) be the
  right-marked RPP in \( \RMRPP(T_0) \) such that, assuming \(
  (i,j),(i,j-1)\in\lm \), the cell \( (i,j) \) is marked in \( \varphi(T) \) if
  and only if the cell \( (i+1,j-1) \) is marked in \( T \) if \(
  (i+1,j-1)\in\lambda \), and the cell \( (\mu'_{j-1}+1, j-1) \) is marked in \(
  T \) otherwise. In other words, the map \( \varphi \) first moves each marking to
  its upper cell with the identification \( (\lambda'_j,j)=(\mu'_j,j) \) of
  the last cells in each column of \( \lambda \) and \( \mu \),
  and then moves to the right cell. Then using
  \eqref{eq:wt_left} and \eqref{eq:wt_right}, one can directly check that \(
  \wt_L(T) = \wt_R(\varphi(T)) \). Clearly, the map \( \varphi \) is bijective,
  and the proof follows. See Figure~\ref{fig:LRMPP}.
\end{proof}

\begin{figure}
  \[
    \ytableausetup{aligntableaux = center}
  \begin{ytableau}
      \none & \none & \none & k^* & k \\
      \none & \none & *(lightgray) k^*\tikzmark{a} & \tikzmark{y}k^* & k \\
      \none & \none & *(lightgray)k\tikzmark{x} & \tikzmark{d}k \\
      \none & k^* & *(lightgray) k^*\tikzmark{c} & \tikzmark{f}k \\
      \none & k & *(lightgray) k^*\tikzmark{e} & \tikzmark{b}k \\
      k^* & k & k
  \end{ytableau} \quad \xrightarrow{\quad\varphi\quad}
    \begin{ytableau}
      \none & \none & \none & k & k^* \\
      \none & \none & k & *(lightgray) k & k^* \\
      \none & \none & k &  *(lightgray) k^* \\
      \none & k & k & *(lightgray) k^* \\
      \none & k & k & *(lightgray) k^* \\
      k & k^* & k^*
    \end{ytableau}
  \begin{tikzpicture}[overlay, remember picture, shorten >=.5pt, shorten <=.5pt, transform canvas={yshift=.25\baselineskip}]
    \draw [->] ({pic cs:x}) -- ({pic cs:y});
    \draw [->] ({pic cs:a}) -- ({pic cs:b});
    \draw [->] ({pic cs:c}) -- ({pic cs:d});
    \draw [->] ({pic cs:e}) -- ({pic cs:f});
  \end{tikzpicture}
  \]
  \caption{ An example of the map \( \varphi \) sending \( T \) to \( \varphi(T)
    \) for \( \lambda=(5,5,4,4,4,3), \mu=(3,2,2,1,1) \), and \( T \) with all
    cells filled with \( k \) or \( k^* \). The arrows on the left diagram show
    how the marked integers and unmarked integers move from column~3 to
    column~4. }
  \label{fig:LRMPP}
\end{figure}

For \( \vr,\vs\in \PP^n \), we denote by \( \RMRPP^{\col(\vr,\vs)}
(\lambda'/\mu') \) the set of right-marked RPPs \( T \)
of shape \( \lambda'/\mu' \) such that
for each \( (i,j)\in \lambda'/\mu' \), \( r_j \le T(i,j) \le s_j \).
According to Proposition~\ref{prop:leftmarked=rightmarked}, 
we can restate \eqref{eq:g_col'} as
\begin{multline} \label{eq:rMRPP=det1}
  \sum_{T\in\RMRPP^{\col(\vr,\vs)}(\lambda'/\mu')} \wt_R(T) \\
  = \det\left( \chi(r_j\le s_i)   e_{\lambda_i-\mu_j-i+j}
    [X_{[r_j,s_i]}-A_{i-1}+A_{j-1}+B_{\lambda_i-1} - B_{\mu_j}]
  \right)_{i,j=1}^n.
\end{multline}

To apply the inductive arguments in the previous section, we consider the
transposes of right-marked RPPs so that the column bound conditions \( r_j \le
T(i,j) \le s_j \) are changed into the row bound conditions \( r_i\le T(i,j) \le
s_i \). 
Since the transpose of a right-marked RPP is a bottom-marked RPP,
we define \( \BMRPP^{\row(\vr,\vs)}(\lambda/\mu) \) similarly to
\( \MRPP^{\row(\vr,\vs)}(\lambda/\mu) \),
and then we can restate \eqref{eq:rMRPP=det1} as follows:
\begin{multline}  \label{eq:rMRPP=det2}
  \sum_{T\in\BMRPP^{\row(\vr,\vs)}(\lambda/\mu)} \wt_B(T)
  \\   = \det\left( \chi(r_j\le s_i)   e_{\lambda_i-\mu_j-i+j}
    [X_{[r_j,s_i]}-A_{i-1}+A_{j-1}+B_{\lambda_i-1} - B_{\mu_j}]
  \right)_{i,j=1}^n.
\end{multline}

We also introduce a notion similar to \( \MRPP^{\row(\vr, \vs)}_I(\lambda/\mu) \).

\begin{defn}\label{defn:BMRPPI}
  Let \( \lambda\in\DPar_n \) and \( \mu\in\Par_n \) with \(
  \mu\subseteq\lambda \), and let \( \vr,\vs\in\PP^n \) and \( I\subseteq [n] \). If
  \( \vr\not\le \vs \), we define \( \BMRPP^{\row(\vr, \vs)}_I(\lambda/\mu)=\emptyset
  \). If \( \vr\le \vs \), we define \( \BMRPP^{\row(\vr, \vs)}_I(\lm) \) to be
  the set of RPPs \( T \) of shape \( \lm \) satisfying the following conditions.
    \begin{itemize}
    \item For all \( (i,j)\in\lm \), we have \( r_i \le T(i,j) \le s_i \).
\item If \( (k, \lambda_k) \) is the minimal cell of \( \lambda \) and \( k-1\in I \),
  then \( T(k,\lambda_k)\ge s_{k-1} \).
    \item Every entry $T(i,j)$ with $T(i,j)=T(i-1,j)$ may be marked, where for each \( i\in [n] \) we define
  \[
    T(i,\lambda_i+1) =
    \begin{cases}
      \infty & \mbox{if \( i\not\in I \)},\\
      s_i & \mbox{if \( i\in I \).}
    \end{cases}
  \]
    \end{itemize}

For \( T\in \BMRPP^{\row(\vr,\vs)}_I(\lm) \), we define
\[
  \wt_B(T) = \prod_{(i,j)\in\lm} \wt_B(T(i,j)),
\]
where
\begin{align*}
  \wt_B(T(i,j)) =
  \begin{cases}
    -\alpha_{i-1} & \mbox{if $T(i,j)$ is marked,}\\
    \beta_{j} & \mbox{if $T(i,j)$ is not marked and $T(i,j)= T(i,j+1)$,}\\
    x_{T(i,j)} & \mbox{otherwise.}
  \end{cases}
\end{align*}
\end{defn}

Finally we state a generalization of \eqref{eq:rMRPP=det2}, which has been shown
to be equivalent to the Jacobi--Trudi-like identity in \eqref{eq:g_col}.

\begin{prop}\label{prop:MRPPI_e}
  Let \( \lambda\in\DPar_n,\mu\in\Par_n, \vr,\vs\in \PP^n \) with \(
  \mu\subseteq\lambda \) such that \( r_i\le r_{i+1} \) and \( s_i\le s_{i+1} \)
  whenever \( \mu_{i}<\lambda_{i+1} \). Suppose that \( (k,\lambda_k) \) is the
  minimal cell of \( \lambda \) and \( I = \{1,2,\dots,p\} \) or \( I =
  \{1,2,\dots,p\}\setminus\{k\} \) for some integer \( k\le p\le \ell(\lambda)
  \) satisfying
  \[
    \lambda_1+1=\dots=\lambda_{k-1}+1=\lambda_k=\lambda_{k+1}=\dots=\lambda_p.
  \]
  Then
  \begin{equation}\label{eq:refined_conj_e}
    \sum_{T\in\BMRPP^{\row(\vr, \vs)}_I(\lm)} \wt_B(T) = 
    \det\left( e_{\lambda,\mu}^{\vr,\vs}(I;i,j)  \right)_{i,j=1}^n,
  \end{equation} 
  where
  \[
    e_{\lambda,\mu}^{\vr,\vs}(I;i,j) = \chi(r_j\le s_i)
    e_{\lambda_i-\mu_j-i+j}
    [X_{[r_j,s_i-\chi(i\in I)]}-A_{i-1}+A_{j-1}+B_{\lambda_i-1+\chi(i\in I)} - B_{\mu_j}].
  \]
\end{prop}

Observe that \eqref{eq:rMRPP=det2} is the special case of
Proposition~\ref{prop:MRPPI_e} when \( \lambda \) is a partition and \(
I=\emptyset \). Similar to the previous subsection we prove
Proposition~\ref{prop:MRPPI_e} by showing that both sides of the equation have
the same recursions and initial conditions. Let
\begin{align*}
  M'(\lambda,\mu,\vr,\vs,I) &= \sum_{T\in\BMRPP^{\row(\vr, \vs)}_I(\lm)} \wt_B(T),\\
  D'(\lambda,\mu,\vr,\vs,I) &=\det\left( e_{\lambda,\mu}^{\vr,\vs}(I;i,j)  \right)_{i,j=1}^n.
\end{align*}

The next two lemmas give the same initial conditions for \( M' \) and \( D' \).
These lemmas can be proved by the same arguments in the proofs of
Lemmas~\ref{lem:LHS=RHS=0} and \ref{lem:LHS=RHS=1}, respectively.

\begin{lem}\label{lem:LHS=RHS=0e}
  Let \( \lambda\in\DPar_n,\mu\in\Par_n,  \vr, \vs\in\RPar_n\) with \( \mu\subseteq\lambda \).
  If \( \vr\not\le\vs \), then
  \[
    M'(\lambda,\mu,\vr,\vs,I) = D'(\lambda,\mu,\vr,\vs,I)= 0.
  \]  
\end{lem}

\begin{lem}\label{lem:LHS=RHS=1e}
Let \( \mu\in\Par_n, \vr, \vs\in \PP^n\).
 If \( \vr\le \vs \), then
  \[
    M'(\mu,\mu,\vr,\vs,I) =D'(\mu,\mu,\vr,\vs,I)= 1.
  \]
\end{lem}

The following two lemmas give the same recursions for \( M' \) and \( D' \).

\begin{lem}  \label{lem:rec1-dual_e}
  Let \( \lambda\in\DPar_n,\mu\in\Par_n,\vr,\vs\in\PP^n \) with \( \mu\subseteq\lambda \).
  Suppose that $1\le t\le n-1$ is an integer such that $\mu_t\ge\lambda_{t+1}$.
  Let \( \phi_{\va} \) be the map defined by shifting \( \alpha_i \) to \( \alpha_{i+1}
  \) for \( i\ge1 \), then for \( F\in\{M,D\} \), we have
  \[
    F'(\lambda,\mu,\vr,\vs,I)
    = F'(\lambda^{(1)},\mu^{(1)},\vr^{(1)},\vs^{(1)},I^{(1)})
    \phi_{\va}^t(F'(\lambda^{(2)},\mu^{(2)},\vr^{(2)},\vs^{(2)},I^{(2)})),
  \]
    where \( \gamma^{(1)} = (\gamma_1,\dots,\gamma_t) \) and
    \( \gamma^{(2)} = (\gamma_{t+1},\dots,\gamma_n) \)
    for each \( \gamma\in \{\lambda,\mu,\vr,\vs\} \),
    and \( I^{(1)}= I\cap\{1,2,\dots,t\} \)
    and \( I^{(2)}= I\cap\{t+1,t+2,\dots,n\} \).
\end{lem}
\begin{proof}
  The proof follows from the same argument of the proof of Lemma~\ref{lem:rec1-dual}.
\end{proof}

\begin{lem}\label{lem:MRPP_rec_e}
  Let \( \lambda\in\DPar_n \) be a dented partition with the minimal cell
  \( (k,\lambda_k) \) and let \( \mu\in\Par_n \) be a partition with
  \( \mu\subseteq\lambda\setminus\{(k,\lambda_k)\} \).
  Let \( \vr,\vs\in \RPar_n \) with \( \vr\le\vs \).
  Suppose that \( I = \{1,2,\dots,p\} \) or \( I = \{1,2,\dots,p\}\setminus\{k\} \)
  for some integer \( k\le p\le \ell(\lambda) \) satisfying
  \[
    \lambda_1+1=\dots=\lambda_{k-1}+1=\lambda_k=\lambda_{k+1}=\dots=\lambda_p.
  \]
     If \( \lambda \) is not a partition and \( s_{k-1}=s_k \), then
     for \( F\in\{M,D\} \),
      \[
        F'(\lambda,\mu,\vr,\vs,I)
        =(x_{s_k}-\alpha_{k-1}+\chi(k\in I)\left( \beta_{\lambda_k} - x_{s_k}\right)) F'(\lambda-\epsilon_k,\mu,\vr,\vs,I\cup\{k\}).
      \]
     If \( \lambda \) is a partition or \( s_{k-1}<s_k \), then
     for \( F\in\{M,D\} \),
      \begin{multline*}
        F'(\lambda,\mu,\vr,\vs,I)
        =(x_{s_k}+\chi(k\in I)\left( \beta_{\lambda_k} - x_{s_k}\right)) F'(\lambda-\epsilon_k,\mu,\vr,\vs,I\cup\{k\}) \\
        +F'(\lambda,\mu,\vr,\vs-\epsilon_k,I\setminus\{k\}).
      \end{multline*}
\end{lem}

\begin{proof}
  The case \( F=M \) can be proved similarly as in the proof of
  Lemma~\ref{lem:rec_main1}. Hence, we only consider the case
  \( F=D \).

  For the first assertion, suppose that \( \lambda \) is not a partition and \(
  s_{k-1}=s_k \). Following the argument in the proof of
  Lemma~\ref{lem:rec_main1}, we only need to prove that for \( 1\le j\le n \),
  \begin{equation}  \label{eq:eIjk-claim}
    e_{\lambda,\mu}^{\vr,\vs}(I;k,j) = e_{\lambda,\mu}^{\vr,\vs}(I;k-1,j) + (x_{s_k}-\alpha_{k-1}+\chi(k\in I)(\beta_{\lambda_k}-x_{s_k})) e_{\lambda-\epsilon_k,\mu}^{\vr,\vs}(I\cup\{k\};k,j),
  \end{equation}
  which is an analogue of \eqref{eq:hIkj-claim}.
  Let
  \[
    m=\lambda_k-\mu_j+j-k \qand
    Z=X_{[r_j,s_k]}-A_{k-1}+A_{j-1}+B_{\lambda_k-1}-B_{\mu_j},  
  \]
  then, since \( s_{k-1}=s_k \)  and \(  \lambda_{k-1}=\lambda_k-1 \), we can rewrite each term in \eqref{eq:eIjk-claim} as
  \begin{align*}
    e_{\lambda,\mu}^{\vr,\vs}(I;k,j)
      &= \chi(r_j\le s_k) e_m[Z+\chi(k\in I)(\beta_{\lambda_k}-x_{s_k})], \\
    e_{\lambda,\mu}^{\vr,\vs}(I;k-1,j)
      &= \chi(r_j\le s_k) e_m[Z-x_{s_k}+\alpha_{k-1}], \\
    e_{\lambda-\epsilon_k,\mu}^{\vr,\vs}(I\cup\{k\};k,j)
      &= \chi(r_j\le s_k) e_{m-1}[Z-x_{s_k}].
  \end{align*}
  By iterating Lemma~\ref{lem:e_m[Z]},
\[
  e_m[Z+\beta_{\lambda_k}-x_{s_k}]
  = e_m[Z-x_{s_k}] +\beta_{\lambda_k}e_{m-1}[Z- x_{s_k}]
  = e_m[Z]+(\beta_{\lambda_k}-x_{s_k}) e_{m-1}[Z- x_{s_k}],
\]
  and thus we can write
  \begin{equation} \label{eq:e_m[Z+chi]}
    e_m[Z+\chi(k\in I)(\beta_{\lambda_k}-x_{s_k})]
      = e_m[Z]+\chi(k\in I) (\beta_{\lambda_k}-x_{s_k}) e_{m-1}[Z- x_{s_k}].
  \end{equation}
  By iterating Lemma~\ref{lem:e_m[Z]}, we similarly obtain
  \[
    e_m[Z] = e_m[Z-x_{s_k}+\alpha_{k-1}] +
    (x_{s_k}-\alpha_{k-1})e_{m-1}[Z-x_{s_k}].
  \] 
  By \eqref{eq:e_m[Z+chi]} and the above equation,
  \( e_{\lambda,\mu}^{\vr,\vs}(I;k,j) \) is equal to 
  \begin{align*}
      &\chi(r_j\le s_k) \left( e_m[Z] + \chi(k\in I)(\beta_{\lambda_k}-x_{s_k})e_{m-1}[Z-x_{s_k}] \right) \\
      &= \chi(r_j\le s_k) \left( e_m[Z-x_{s_k}+\alpha_{k-1}] + (x_{s_k}-\alpha_{k-1}+\chi(k\in I)(\beta_{\lambda_k}-x_{s_k}))e_{m-1}[Z-x_{s_k}] \right) \\
      &= e_{\lambda,\mu}^{\vr,\vs}(I;k-1,j) + (x_{s_k}-\alpha_{k-1} + \chi(k\in I)(\beta_{\lambda_k}-x_{s_k})) e_{\lambda-\epsilon_k,\mu}^{\vr,\vs}(I\cup\{k\};k,j),
  \end{align*}
  which shows the first assertion.
  
  For the second assertion, suppose that \( \lambda \) is a partition or \( s_{k-1} < s_k \). Similar to the first case, it suffices to show that for \( 1\le j\le n \),
  \begin{equation}\label{eq:eIkj-claim2}
    e_{\lambda,\mu}^{\vr,\vs}(I;k,j) = e_{\lambda,\mu}^{\vr,\vs-\epsilon_k}(I\setminus\{k\};k,j)
      + (x_{s_k}+\chi(k\in I)(\beta_{\lambda_k}-x_{s_k})) e_{\lambda-\epsilon_k,\mu}^{\vr,\vs}(I\cup\{k\};k,j).
  \end{equation}
  Using the terms \( Z \) and \( m \) defined above, we can write
  \begin{align*}
    e_{\lambda,\mu}^{\vr,\vs}(I;k,j)
    &= \chi(r_j\le s_k) e_m[Z+\chi(k\in I)(\beta_{\lambda_k}-x_{s_k})], \\
    e_{\lambda,\mu}^{\vr,\vs-\epsilon_k}(I\setminus\{k\};k,j)
      &= \chi(r_j\le s_k-1) e_m[Z-x_{s_k}], \\
    e_{\lambda-\epsilon_k,\mu}^{\vr,\vs}(I\cup\{k\};k,j) 
      &= \chi(r_j\le s_k) e_{m-1}[Z-x_{s_k}].
  \end{align*}
  Then by \eqref{eq:e_m[Z+chi]} and Lemma~\ref{lem:e_m[Z]}, we obtain
  \begin{align*}
    e_{\lambda,\mu}^{\vr,\vs}(I;k,j)
      &= \chi(r_j\le s_k) \left( e_m[Z] + \chi(k\in I)(\beta_{\lambda_k}-x_{s_k}) e_m[Z-x_{s_k}] \right) \\
      &= \chi(r_j\le s_k) \left( e_m[Z-x_{s_k}] + (x_{s_k}+\chi(k\in I)(\beta_{\lambda_k}-x_{s_k}))e_{m-1}[Z-x_{s_k}] \right) \\
      &= \chi(r_j\le s_k) e_m[Z-x_{s_k}] + (x_{s_k}+\chi(k\in I)(\beta_{\lambda_k}-x_{s_k})) e_{\lambda-\epsilon_k,\mu}^{\vr,\vs}(I\cup\{k\};k,j).
  \end{align*}
  Hence, to prove \eqref{eq:eIkj-claim2} it is enough to show
  \[
    \chi(r_j\le s_k-1) e_m[Z-x_{s_k}] = \chi(r_j\le s_k) e_m[Z-x_{s_k}].
  \]
  Obviously, if \( r_j\neq s_k \), then the both sides are the same. If \( r_j=s_k \), the above equation is equivalent to
  \[
    e_{\lambda_k-\mu_j+j-k}[-A_{k-1}+A_{j-1}+B_{\lambda_k-1}-B_{\mu_j}] = 0.
  \]
  This can be prove similarly as \eqref{eq:chi-sk-1}, and thus we obtain \eqref{eq:eIkj-claim2}.
\end{proof}

Finally we are ready to prove Proposition~\ref{prop:MRPPI_e}. This can be done
by the same arguments in the proof of Proposition~\ref{prop:MRPPI} with
Lemmas~\ref{lem:LHS=RHS=0e}, \ref{lem:LHS=RHS=1e}, \ref{lem:MRPP_rec_e}, and
\ref{lem:rec1-dual_e} in place of Lemmas~\ref{lem:LHS=RHS=0},
\ref{lem:LHS=RHS=1}, \ref{lem:rec_main1} and \ref{lem:rec1-dual}.
We omit the details.

\section*{Acknowledgments}
The authors thank Travis Scrimshaw for many helpful comments and fruitful discussions, and Darij Grinberg for valuable feedback. They are also grateful to the anonymous referees for their careful reading and thoughtful suggestions.

The authors were supported by the National Research Foundation of
Korea (NRF) grants \#2022R1A2C101100911 and \#2016R1A5A10080.
Byung-Hak Hwang was supported by a KIAS Individual Grant (MG098201) at Korea Institute for Advanced Study.
Minho Song was supported by the National Research Foundation of Korea (NRF)
grant funded by the Korea government (MSIT) (No. 2022R1C1C2009025).

\appendix
\section{Marked rim border tableaux}
\label{sec:RBT}

In this appendix, we introduce marked rim border tableaux and provide
a weight-preserving bijection between marked rim border tableaux and
marked RPPs. This shows that Theorem~\ref{thm:MRPP} generalizes
Yeliussizov's combinatorial model for the dual Grothendieck polynomial
\( g^{(\alpha,\beta)}_\lambda(\vx) \).

  \begin{defn}\label{def:marked_rim}
    For a reverse plane partition \( T \), let \( T_i \) denote the
    part of \( T \) consisting of cells with entry \( i \). Define the
    \emph{border} \( R_i \) of \( T_i \) as the part consisting of
    cells \( c \) where the upper-left cell of \( c \) is not in
    \( T_i \). The \emph{inner part} of \( T_i \) is defined to be
    \( T_i\setminus R_i \). Each border \( R_i \) can be partitioned
    by cutting along a vertical side of a cell. The resulting cut
    blocks are called \emph{rim hooks}. A \emph{marked rim border
      tableau} is a reverse plane partition \( T \) in which for each
    \( i \), the border \( R_i \) of \( T_i \) is partitioned into rim
    hooks and each cell in the inner part \( T_i\setminus R_i \) can
    be marked. We denote by \( \MRBT(\lm) \) the set of marked rim
    border tableaux of shape \( \lm \).
  \end{defn}
 
  \begin{defn}\label{def:weight_marked_rim}
    For a marked rim border tableau \( T \), we define the \emph{weight} \( \wt(T) \)
    to be the product of \( \wt(c) \) for all cells \( c \) in \( T \),
    where \( \wt(c) \) is defined as follows.
    \begin{itemize}
    \item If \( c \) is in a rim hook \( R \) with entry \( i \), then
      \[ \wt(c) = 
      \begin{cases}
        \alpha & \mbox{ if the right cell of \( c \) is in \( R \), }\\
        \beta  & \mbox{ if the upper cell of \( c \) is in \( R \), }\\
        x_i & \mbox{ otherwise. } 
      \end{cases}
    \]
  \item If \( c \) is in an inner part, then
      \[ \wt(c) = 
      \begin{cases}
        \alpha & \mbox{ if \( c \) is marked, }\\
        \beta  & \mbox{ if \( c \) is not marked. }
      \end{cases}
    \]
    \end{itemize}
  \end{defn}
 
  Yeliussizov's result can be restated as follows.

\begin{thm}\cite[Theorem~7.2]{Yeliussizov2017}
   Let \( \va = (-\alpha,-\alpha,\dots) \) and
   \( \vb = (\beta,\beta,\dots) \).
   Then
\[
    g_\lambda(\vx;\va,\vb) = \sum_{T\in\MRBT(\lambda)} \wt(T).
\]
\end{thm}
  
The following proposition with \( \mu = \emptyset \) shows that
Theorem~\ref{thm:MRPP} generalizes the above result of Yeliussizov.
  
   \begin{prop}\label{pro:bij_rim_RPP}
   Let \( \va = (-\alpha,-\alpha,\dots) \) and
   \( \vb = (\beta,\beta,\dots) \).
   Then, there is a weight-preserving
   bijection between \( \MRBT(\lm) \) and
   \( \MRPP(\lm) \).
 \end{prop}
 \begin{proof}
   We establish a bijection \( \varphi: \MRBT(\lm)\to \MRPP(\lm) \) as
   follows. Let \( T\in \MRBT(\lm) \). As before, let \( T_i \) be the
   part of \( T \) consisting of cells with entry \( i \), and let
   \( R_i \) be its border.

   For each \( i \), we shift all cells in the inner part
   \( T_i\setminus R_i \) to the left by one. Next, we move each
   overlapping cell in \( R_i \) to the rightmost cell in the same row
   in \( T_i \). Then we assign a mark to each cell with
   \( \alpha \). We define \( \varphi(T) \) to be the resulting
   tableau; see~\Cref{fig:rim}.

   Note that all overlapping cells in \( R_i \) in the process are in
   a rim border with weight \( \beta \). Thus, the resulting tableau
   \( \varphi(T) \) is a marked RPP. One can see that this gives a
   weight-preserving bijection.
 \end{proof}
 
  \begin{figure}
    \centering
\begin{tikzpicture}[scale=0.6, every node/.style={minimum size=1.5em, font=\large, anchor=center}]
  \foreach \x/\y in {
    6/0,                 
    6/1,               
    3/2, 4/2, 5/2, 6/2, 
    1/3, 2/3, 3/3,  
    1/4,       
    1/5      
  }{
    \fill[gray!20] (\x, -\y) rectangle ++(1, -1); 
  }
  \foreach \x/\y/\label in {
    6/0/\( x_i \),                 
    6/1/\( \beta \),               
    3/2/\( \alpha \), 4/2/\( x_i \), 5/2/\( \alpha \), 6/2/\( \beta \), 
    1/3/\( x_i \), 2/3/\( \alpha \), 3/3/\( \beta \), 4/3/\( \beta \), 5/3/\( \alpha \), 6/3/\( \alpha \),
    1/4/\( \beta \),2/4/\( \alpha \),3/4/\( \beta \), 4/4/\( \beta \),5/4/\( \alpha \),
    1/5/\( \beta \),2/5/\( \alpha \),3/5/\( \alpha \),4/5/\( \alpha \), 5/5/\( \beta \)
  }{
    \node at (\x+0.5, -\y-0.5) {\label}; 
  }
  \draw[thick] (1,-6) -- (1,-3) -- (3,-3) -- (3,-2) -- (6,-2) -- (6,0) -- (7,0) -- (7,-4) -- (6,-4) -- (6,-6) -- (1,-6);
  \draw[thick] (2,-6) -- (2,-3);
  \draw[thick] (2,-4) -- (4,-4) -- (4,-3) -- (5,-3) --(5,-2);
  \draw[thick] (5,-3) -- (7,-3);
  \begin{scope}[xshift=10cm]
\foreach \x/\y in {
    6/0,                 
    6/1,               
    3/2, 4/2, 5/2, 6/2, 
    1/3, 2/3, 6/3,  
    5/4,       
    5/5      
  }{
    \fill[gray!20] (\x, - \y) rectangle ++(1, -1); 
  }
  \foreach \x/\y/\label in {
    6/0/\( x_i \),                 
    6/1/\( \beta \),               
    3/2/\( \alpha \), 4/2/\( x_i \), 5/2/\( \alpha \), 6/2/\( \beta \), 
    1/3/\( x_i \), 2/3/\( \alpha \), 6/3/\( \beta \), 3/3/\( \beta \), 4/3/\( \alpha \), 5/3/\( \alpha \),
    1/4/\( \alpha \),2/4/\( \beta \), 3/4/\( \beta \),4/4/\( \alpha \), 5/4/\( \beta \),
    1/5/\( \alpha \),2/5/\( \alpha \),3/5/\( \alpha \),4/5/\( \beta \), 5/5/\( \beta \)
  }{
    \node at (\x+0.5, -\y-0.5) {\label}; 
  }
  \draw[thick] (1,-6) -- (1,-3) -- (3,-3) -- (3,-2) -- (6,-2) -- (6,0) -- (7,0) -- (7,-4) -- (6,-4) -- (6,-6) -- (1,-6);
  \draw[thick] (1,-4) -- (2,-4) -- (2,-3);
  \draw[thick] (2,-4) -- (3,-4) -- (3,-3) -- (5,-3) --(5,-2);
  \draw[thick] (5,-3) -- (7,-3);
  \draw[thick] (6,-4) --(6,-3);
  \draw[thick] (5,-6) --(5,-4) -- (6,-4);
  \end{scope}
\end{tikzpicture}
\caption{ The left tableau shows the part \( T_i \) in a rim border
  tableau \( T \), where the border is gray and the inner part is
  white. Each inner cell has weight \( \alpha \) if marked, \( \beta \)
  otherwise. The right shows the corresponding part with entry \( i \)
  in the marked RPP \( \varphi(T) \), where the marked cells are
  those containing an \( \alpha \).}
    \label{fig:rim}
  \end{figure}

\bibliographystyle{abbrv}

\end{document}